\documentclass{amsart}
\usepackage{amsmath,amsfonts,amssymb,amsthm,epsfig,amscd,epsfig,psfrag,comment}
\usepackage[all]{xy}

\newtheorem{Theorem}{Theorem}[section]
\newtheorem{Proposition}[Theorem]{Proposition} 
\newtheorem{Lemma}[Theorem]{Lemma}
\newtheorem{Definition}[Theorem]{Definition}

\newtheorem{Corollary}[Theorem]{Corollary}

\theoremstyle{definition}

\newtheorem{Remark}[Theorem]{Remark}

\newcommand{\p}{{\mathbb{P}}}
\newcommand{\Z}{\mathbb{Z}}
\newcommand{\C}{\mathbb{C}}
\def\bG{{\mathbb{G}}}

\def\At{{\mbox{At}}}
\def\la{\langle}
\def\ra{\rangle}
\def\hA{\hat{A}}
\def\hB{\hat{B}}
\def\hC{\hat{C}}
\def\hD{\hat{D}}
\def\l{\lambda}
\def\trace{\mbox{Trace}}
\def\Spec{\mbox{Spec}}
\def\id{{\mathrm{id}}}
\def\k{{\Bbbk}}

\def\sl{{\mathfrak{sl}}}

\def\D{{\mathcal{D}}}
\def\E{{\sf{E}}}
\def\F{{\sf{F}}}
\def\Fl{{\mathbb{F}l}}

\def\sF{{\mathcal{F}}}

\def\bG{{\mathbb G}}
\def\sG{{\mathcal{G}}}

\def\tY{{\tilde{Y}}}

\def\bA{{\mathbb{A}}}

\def\sP{{\mathcal{P}}}

\def\sQ{{\mathcal{Q}}}

\newcommand{\Cone}{\mathrm{Cone}}
\def\O{{\mathcal O}}
\def\sE{{\mathcal{E}}}
\def\H{{\mathcal{H}}}

\def\dim{\mbox{dim}}

\newcommand{\DMod}{\operatorname{D-mod}}
\newcommand{\modu}{\operatorname{-gr mod}}

\DeclareMathOperator{\Hom}{Hom}
\DeclareMathOperator{\Ext}{Ext}
\DeclareMathOperator{\supp}{supp}
\DeclareMathOperator{\End}{End}
\DeclareMathOperator{\Aut}{Aut}

\begin{document}

\title{Coherent sheaves and categorical $\sl_2$ actions}

\author{Sabin Cautis}
\email{scautis@math.harvard.edu}
\address{Department of Mathematics\\ Rice University \\ Houston, TX}

\author{Joel Kamnitzer}
\email{jkamnitz@math.toronto.edu}
\address{Department of Mathematics\\ University of Toronto \\ Toronto, ON Canada}

\author{Anthony Licata}
\email{amlicata@math.stanford.edu}
\address{Department of Mathematics\\ Stanford University \\ Palo Alto, CA}

\begin{abstract}
We introduce the concept of a geometric categorical $\sl_2$ action and relate it to that of a strong categorical 
$\sl_2$ action. The latter is a special kind of 2-representation in the sense of Rouquier. The main result is that a geometric categorical $\sl_2$ action induces a strong categorical $\sl_2$ action. This allows one to apply the theory of strong $\sl_2$ actions to various geometric situations. Our main example is the construction of a geometric categorical $\sl_2$ action on the derived category of coherent sheaves on cotangent bundles of Grassmannians. 
\end{abstract}

\date{\today}
\maketitle
\tableofcontents

\section{Introduction}

\subsection{Actions of $ \sl_2$ on categories}
An action of $\sl_2 $ on a finite-dimensional vector space $ V $ consists of a direct sum decomposition $ V = \oplus V(\lambda) $ into weight spaces and linear maps $ E(\lambda) : V(\lambda -1) \rightarrow V(\lambda + 1) $ and $ F(\lambda) : V(\lambda +1) \rightarrow V(\l -1) $.   These maps satisfy the condition
\begin{equation} \label{eq:commutation}
E(\lambda-1) F(\lambda -1) - F(\l+1) E(\l+1) = \l \id_{V(\l)}.
\end{equation}
Such an action automatically integrates to the group $ SL_2 $.  In particular, the reflection element $ t = \bigl[ \begin{smallmatrix} \phantom{-}0 &1 \\ -1& 0 \end{smallmatrix} \bigr] \in SL_2$ acts on $ V $, inducing an isomorphism $ V(-\lambda) \rightarrow V(\lambda) $.

A na\"ive categorical action of $ \sl_2 $ consists of a sequence of categories $ \D(\l) $ with functors
$$ \E(\l) : \D(\l-1) \rightarrow \D(\l +1) \text{ and } \F(\l) : \D(\l +1) \rightarrow \D(\l -1)$$ 
between them. These functors should satisfy a categorical version of (\ref{eq:commutation}) above, 
\begin{equation} \label{eq:catcommutation}
\E(\lambda-1) \circ \F(\lambda -1) \cong \id_{\D(\l)}^{\oplus \l} \oplus \F(\l+1) \circ \E(\l+1), \quad  \text{ for }  \l \ge 0 
\end{equation}
and an analogous condition when $ \l \le 0 $. This is just a na\"ive notion of categorical $ \sl_2 $ action, since ideally there should be morphisms between the functors which induce the isomorphisms (\ref{eq:catcommutation}).

The purpose of this paper and the accompaning papers \cite{ckl1}, \cite{ckl2}, is to apply categorical $\sl_2 $ actions to the geometric situation where $ \D(\l) $ is the derived category of coherent sheaves of a variety.  The main example discussed in this paper is the case of cotangent bundles to Grassmannians of planes in a fixed $N$ dimensional vector space. Thus we fix $N$ and as $ k $ varies we let  $ \D(N-2k) := DCoh(T^\star(\bG(k,N))$.

\subsection{Strong categorical $\sl_2 $ actions and geometric categorical $\sl_2 $ actions}

In this paper we have two main definitions. First we define the notion of a \emph{strong categorical action} of $ \sl_2 $ (section \ref{se:defstrong}), a modification of definitions due to Chuang-Rouquier \cite{CR}, Lauda \cite{l1}, and Rouquier \cite{R}.  These axioms include the additional data of morphisms of functors $ X : \E \rightarrow \E$  and $ T : \E^2 \rightarrow \E^2 $ which rigidify the isomorphisms (\ref{eq:catcommutation}).  In a companion paper, \cite{ckl2}, where we prove (using ideas from \cite{CR}) that whenever there is a strong categorical action of $\sl_2$ whose weight spaces are triangulated categories, then we can construct a triangulated equivalence between $ \D(\l) $ and $ \D(-\l) $.

The second notion introduced in this paper is that of a \emph{geometric categorical} $\sl_2 $ action (section \ref{se:defgeom}). This means that we have a sequence of varieties $ Y(\l) $ and Fourier-Mukai kernels $ \sE(\l), \sF(\l) $, which are objects in the derived categories of the products $ Y(\l-1) \times Y(\l+1) $.  These kernels are required to satisfy the commutation relation (\ref{eq:catcommutation}), but only at the level of cohomology. We also demand that there exist certain deformations $ \tY(\l) \rightarrow \bA^1$ of $ Y(\l) $ with some special properties.  The idea to impose the existence of deformations was inspired by the work of Huybrechts-Thomas \cite{HT} (see Remark \ref{re:HT}).

The main theorem of this paper (Theorem \ref{th:main}) is that a geometric categorical $\sl_2$ action gives rise to a strong categorical $ \sl_2 $ action when the categories involved are the derived categories of coherent sheaves $ D(Y(\l)) $ and where the functors $ \E(\l), \F(\l) $ are induced by the kernels $\sE(\l), \sF(\l)$.  Roughly speaking, the morphism $ X: \sE(\l) \rightarrow \sE(\l)[2] $ is the obstruction to deforming $ \sE(\l) $ in the family $ \tY(\l-1) \times_{\bA^1} \tY(\l+1) $.  Sections \ref{se:formality} and \ref{se:proofofnilHecke} are devoted to the proof of this theorem.  In practice it is much easier to check that certain geometric constructions give rise to a geometric categorical $\sl_2$ action rather than a strong categorical $\sl_2$ action. So Theorem \ref{th:main} provides a bridge between geometry (and the results in \cite{ckl1}) and more formal algebraic/categorical constructions provided by a strong $\sl_2$ action (such as the equivalences constructed in \cite{ckl2}). 

\subsection{Relation to 2-categories of Rouquier and Lauda}

In \cite{R}, Rouquier defined a 2-categorical version of quantum $\sl_2$ (based on work in \cite{CR}). This is closely related (not coincidentally) to our definition of strong categorical $ \sl_2 $ action.  A strong categorical $\sl_2$ action immediately gives rise to a 2-functor from Rouquier's 2-category into the 2-category of triangulated categories. Thus another way of viewing Theorem \ref{th:main} is to note that it provides a way of obtaining natural 2-representations of Rouquier's 2-category.  A strong categorical $\sl_2$ action is a slightly more restrictive notion than a 2-representation of Rouquier's 2-category, however, because our definition demands additional conditions on the endomorphism algebras $\Ext(\E(\l)^{(r)},\E(\l)^{(r)})$.  These additional restrictions are used in the proof of the equivalence of triangulated categories between $\D(\l)$ and $\D(-\l)$ considered in \cite{ckl2}.

In \cite{l1} Lauda also constructs a 2-category which categorifies quantum $\sl_2$. Lauda's definition is similar to Rouquier's but involves some extra technical relations which, from our point of view, are not entirely necessary;  by this we mean that the equivalence between $\D(-\l)$ and $\D(\l)$ constructed in \cite{ckl2} does not require these extra relations. It is not obvious that a geometric categorical $\sl_2$ action gives rise to a 2-representation of Lauda's 2-category (although it is natural to conjecture that it does). Understanding better the role played by these extra relations in Lauda's definition is an interesting problem. 

Rouquier \cite{R} and Khovanov-Lauda \cite{KL} have also defined analogous (and closely related) 2-categories for other Kac-Moody Lie algebras.  In a future work \cite{ckl3}, we will construct 2-representations of these 2-categories (in the simply-laced case) on derived categories of coherent sheaves on quiver varieties, generalizing the action on cotangent bundles to Grassmannians described below.

\subsection{Contangent bundles to Grassmannians}
Our main example of a geometric categorical $ \sl_2 $ action is that of cotangent bundles to Grassmannians.  We fix $ N $ and consider $ Y(N-2k) := T^\star \bG(k, N) $ as our varieties.  There are natural correspondence between Grassmannians which give us the kernels $ \sE, \sF $.  In section \ref{sec:mainexample}, we use the results of \cite{ckl1} to prove that this is indeed a geometric categorical $ \sl_2 $ action.  

Hence as a corollary of this paper and of \cite{ckl2}, we obtain an explicit equivalence of triangulated categories between $ DCoh(T^\star \bG(k,N))$ and $ DCoh(T^\star \bG(N-k,N)) $.  This answers an open question raised by papers of Kawamata \cite{kaw2} and Namikawa \cite{nam2} (see \cite{ckl2} for more details).

Let us now informally explain how the example of derived categories of coherent sheaves on cotangent bundles to Grassmannians is related to examples to categorical $ \sl_2 $ actions studied by Chuang-Rouquier \cite{CR}, and Lauda \cite{l1}.  

In section 7.4 of \cite{CR}, Chaung-Rouquier defined a strong categorical $ \sl_2 $ action (in their sense) on the singular blocks of category $ \mathcal{O} $ for $ \mathfrak{gl}_N $, following the work of Bernstein-Frenkel-Khovanov \cite{BFK}. By Koszul duality of Beilinson-Ginzburg-Sorgel \cite{BGS}, this can be considered as a categorical $ \sl_2 $ action on parabolic category $ \O $ for $ \mathfrak{gl}_N$.  We may restrict to the particular parabolic categories considered in \cite{BFK} section 4.  Under the Beilinson-Bernstein localization, these categories are equivalent to the categories of $ B $-equivariant $ D$-modules on the Grassmannians $ \bG(k,N)$.  $D$-modules on $ \bG(k,N) $ are related to coherent sheaves on $ T^* \bG(k,N) $ by taking associated graded.

In section 7.7.2 of \cite{CR}, Chuang-Rouquier defined a strong categorical $ \sl_2 $ action where the categories $ \D(N-2k) $ were categories of modules over $H^*(\bG(k,N))$.  Similarly, in \cite{l1}, Lauda defined a functor from his 2-category to a 2-category whose 1-morphisms are certain $ H^*(\bG(k,N)), H^*(\bG(l,N)) $ graded bimodules. There is a functor from the category of perverse sheaves on $ \bG(k,N) $ (which is equivalent to $ \DMod(\bG(k,N))$) to $ H^*(\bG(k,N)) $ graded modules by taking total cohomology.  

Also there is a functor from the derived category of coherent sheaves on $ T^*(\bG(k,N)) $ to dg modules over $ H^*(\bG(k,N)) $ by taking $ \Ext(\O_{\bG(k,N)}, \cdot)$, because $ \Ext^*(\O_{\bG(k,N)}, \O_{\bG(k,N)}) = H^*(\bG(k,N)) $ (see Remark 5.11 in \cite{ck1}).

To summarize we have the following rough picture:
\begin{equation*}
\xymatrix{
\DMod(\bG(k,N))   \ar[d]_{\text{total cohomology}} \ar[r]^{\text{ass. graded}} & Coh(T^*(\bG(k,N)) \ar[ld]^{\text{Ext into $\O_{\bG(k,N)}$}} \\
 H^*(\bG(k,N))\modu
}
\end{equation*}

We expect that these ``functors'' (if they are made precise as functors) will intertwine the three $ \sl_2 $ actions.

\subsection{Acknowledgements}

We would like to thank Daniel Huybrechts, Mikhail Khovanov, and Raphael Rouquier for helpful conversations.

S.C. was partially supported by National Science Foundation Grant 0801939. He also thanks MSRI for its support and hospitality during the spring of 2009.  J.K. was partially supported by a fellowship from the American Institute of Mathematics. A.L. would like to thank the Max Planck Institute in Bonn for support during the 2008-2009 academic year.

\section{Main definitions and results}\label{sec:maindefs}

First, a bit of notational discussion.  We will denote composition of functors by juxtaposition and reserve the symbol $ \circ $ to denote composition of morphisms.  Also we will denote the identity morphism by $I$ and the identity functor by $ \id  $. We denote by $\bG(k,n)$ the Grassmannian parametrizing $k$-planes in $\C^n$. We denote by $H^\star(\bG(k,n))$ the usual cohomology of $\bG(k,n)$ but shifted so that it is symmetric with respect to degree zero (equivalently, it is the intersection cohomology).  For the purposes of the definition of strong $ \sl_2 $ categorification, we will use $ \la \cdot \ra $ for the grading, whereas later in the paper we will use replace $ \la k \ra $ by $ [k]\{-k\}$.  For example, 
\begin{align*}
H^\star(\p^n) &= \k \la n \ra \oplus \k \la n-2 \ra \oplus \dots \oplus \k \la -n+2 \ra \oplus \k \la -n \ra \\
&=  \k[n]\{-n\} \oplus \k[n-2]\{-n+2\} \oplus \dots \oplus \k[-n+2]\{n-2\} \oplus \k[-n]\{n\}.
\end{align*}
By convention $H^\star(\p^{-1})$ is zero. 

\subsection{Strong $\sl_2$ categorification} \label{se:defstrong}

Let $\k$ be a field. A {\bf strong categorical $\sl_2$ action} consists of the following data.

\begin{enumerate}
\item A sequence of $\k$-linear, $\Z$-graded, additive categories $\D(-N), \dots, \D(N)$ which are idempotent complete. We say that a category is graded it it has a shift functor $\la \cdot \ra$ which is an equivalence.
\item Functors 
$$\E^{(r)}(\l): \D(\l-r) \rightarrow \D(\l+r) \text{ and } \F^{(r)}(\l): \D(\l+r) \rightarrow \D(\l-r)$$
for $r\geq 0$ and $\l \in \Z$. We assume these functors are additive and commute with shift.  We will usually write $\E(\l)$ for $\E^{(1)}(\l)$ and $\F(\l)$ for $\F^{(1)}(\l)$.
It is convenient to set $\E^{(0)}(\l)$ and $\F^{(0)}(\l)$ to be the identity functor $\id$ on $Y(\l)$.
\item Morphisms
$$\eta: \id \rightarrow \F^{(r)}(\l)  \E^{(r)}(\l) \la r \l \ra \text{ and } 
\eta: \id \rightarrow \E^{(r)}(\l)  \F^{(r)}(\l) \la -r\l \ra$$
$$\varepsilon: \F^{(r)}(\l) \E^{(r)}(\l) \rightarrow \id \la r\l \ra \text{ and }
\varepsilon: \E^{(r)}(\l)  \F^{(r)}(\l) \rightarrow \id \la -r\l \ra. $$
\item Morphisms
$$\iota : \E^{(r+1)}(\l) \la r \ra \rightarrow \E(\l+r) \E^{(r)}(\l-1) \text{ and } \pi : \E(\l+r) \E^{(r)}(\l-1) \rightarrow \E^{(r+1)}(\l) \la -r \ra.$$
\item Morphisms
$$X(\l): \E(\l) \la -1 \ra \rightarrow \E(\l) \la 1 \ra \text{ and }
T(\l): \E(\l+1) \E(\l-1) \la 1 \ra \rightarrow \E(\l+1) \E(\l-1) \la -1 \ra.$$
\end{enumerate}

On this data we impose the following additional conditions.

\begin{enumerate}
\item Each (graded piece of the) $\Hom$ space between two objects in $\D(\l)$ is finite dimensional. 
\item The morphism $ \eta $ and $ \varepsilon $ are units and counits of adjunctions
\begin{enumerate}
\item $\E^{(r)}(\l)_R = \F^{(r)}(\l) \la r\l \ra$ for $r \ge 0$
\item $\E^{(r)}(\l)_L = \F^{(r)}(\l) \la -r\l \ra$ for $r \ge 0$
\end{enumerate}
\item We have isomorphisms
$$\E^{(r_2)}(\l+r_1) \E^{(r_1)}(\l-r_2) \cong \E^{(r_1+r_2)}(\l) \otimes_{\k} H^\star(\bG(r_1,r_1+r_2))$$
For example,
$$\E(\l+1) \E(\l-1) \cong \E^{(2)}(\l) \la -1 \ra \oplus \E^{(2)}(\l) \la 1 \ra.$$
In general we do not impose that this isomorphism is induced by a particular natural transformation. However, in the case $r_1=r$ and $r_2=1$ we do require that the maps
$$\oplus_{i=0}^r (X(\l + r)^i  I) \circ \iota \la -2i \ra : \E^{(r+1)}(\l) \otimes_\k H^\star(\p^r) \rightarrow \E(\l+r) \E^{(r)}(\l-1)$$
and
$$\oplus_{i=0}^r \pi \la 2i \ra \circ (X(\l+r)^i I) : \E(\l+r) \E^{(r)}(\l-1) \rightarrow \E^{(r+1)}(\l) \otimes_\k H^\star(\p^r)$$ 
are isomorphisms.  We also have the analogous condition when $r_1=1$ and $r_2=r$. 

\begin{Remark} Intuitively, $\iota$ maps into the ``bottom'' factor of 
$$\E(\l+r)  \E^{(r)}(\l-1) \cong \E^{(r+1)}(\l) \otimes_\k H^\star(\p^r)$$
while $\pi$ maps out of the ``top'' factor. 
\end{Remark}

\item If $ \l \le 0 $ then
\begin{equation*}
\F(\l+1) \E(\l+1) \cong \E(\l-1) \F(\l-1) \oplus \id \otimes_{\k} H^\star(\p^{-\l-1}).
\end{equation*}
The isomorphism is induced by
$$\sigma + \sum_{j=0}^{-\l-1} (I  X(\l+1)^j) \circ \eta: \E(\l-1)  \F(\l-1) \oplus \id \otimes_{\k} H^\star(\p^{-\l-1}) \xrightarrow{\sim} \F(\l+1)  \E(\l+1)$$
where $\sigma$ is the composition of maps
\begin{eqnarray*}
\E(\l-1)  \F(\l-1) & \xrightarrow{\eta I I} & \F(\l+1)  \E(\l+1)  \E(\l-1)  \F(\l-1) \la \l+1 \ra \\
& \xrightarrow{I  T({\l})  I} & \F(\l+1)  \E(\l+1)  \E(\l-1)  \F(\l-1) \la \l-1 \ra \\
& \xrightarrow{I I \epsilon} & \F(\l+1)  \E(\l+1).
\end{eqnarray*}
Similarly, if $\l \ge 0$ then 
\begin{equation*}
\E(\l-1)  \F(\l-1) \cong \F(\l+1)  \E(\l+1) \oplus \id \otimes_{\k} H^\star(\p^{\l-1})
\end{equation*}
with the isomorphism induced in the same way as above. 

\item The $X$s and $T$s satisfy the nil affine Hecke relations:
\begin{enumerate}
\item $T(\l)^2 = 0$
\item $(I  T(\l-1)) \circ (T(\l+1)  I) \circ (I  T(\l-1)) = (T(\l+1) I) \circ (I T(\l-1)) \circ (T(\l+1)  I)$ as endomorphisms of $ \E(\l+2)\E(\l)\E(\l-2)$.
\item $(X(\l+1)  I) \circ T(\l) - T(\l) \circ (I  X(\l-1)) = I = - (I  X(\l-1)) \circ T(\l) + T(\l) \circ (X(\l+1) I)$ as endomorphisms of $ \E(\l+1)\E(\l-1) $.
\end{enumerate}

\item For $r \ge 0$ we have $\Hom(\E^{(r)}(\l), \E^{(r)}(\l) \la i \ra) = 0$ if $i < 0$ while $\End(\E^{(r)}(\l)) = \k \cdot I$. 
\end{enumerate}

Note that all functors appearing in the definition above can be obtained from the functors $\E(\l)$ by composition, taking direct summands and by taking (left or right) adjoints. 

Although we have categories $\D(\lambda)$ corresponding to weights $-N \le \lambda \le N$ the $\E$s and $\F$s jump by an even amount from an odd weight to odd weight or from an even weight to an even weight. So we can separate our analysis into studying the odd and even weights. It will therefore often be  convenient to assume that $\D(-N+1), \D(-N+3), \dots, \D(N-3), \D(N-1)$ are empty. 

\begin{Remark}
A strong categorical $\sl_2 $ action is the same thing as an integrable, graded representation of Rouquier's 2-category in the 2-category of $\k$-linear categories \cite{R}, along with the above extra condition on $\End(\E^{(r)})$. This follows immediately if one compares his definition to ours. 
\end{Remark}

\subsection{Geometric categorical $\sl_2$ action} \label{se:defgeom}

\subsubsection{A few preliminaries}

If $X$ is a variety we denote by $D(X)$ the bounded derived category of coherent sheaves on $X$. An object $\sP \in D(X \times Y)$ whose support is proper over $Y$ induces a Fourier-Mukai (FM) functor $\Phi_{\sP}: D(X) \rightarrow D(Y)$ via $(\cdot) \mapsto \pi_{2*}(\pi_1^* (\cdot) \otimes \sP)$ (where every operation is derived). One says that $\sP$ is the FM kernel which induces $\Phi_{\sP}$. The right and left adjoints $\Phi_{\sP}^R$ and $\Phi_{\sP}^L$ are induced by $\sP_R := \sP^\vee \otimes \pi_2^* \omega_X [\dim(X)]$ and $\sP_L := \sP^\vee \otimes \pi_1^* \omega_Y [\dim(Y)]$ respectively. 

If $\sQ \in D(Y \times Z)$ then $\Phi_{\sQ}  \Phi_{\sP} \cong \Phi_{\sQ * \sP}: D(X) \rightarrow D(Y)$ where $\sQ * \sP = \pi_{13*}(\pi_{12}^* \sP \otimes \pi_{23}^* \sQ)$ is the convolution product (see also \cite{ck1} section 3.1).

If $X$ carries a $\k^\times$ action then we can also consider the bounded derived category of $\k^\times$-equivariant coherent sheaves on $X$. On $X$ the sheaf $\O_X\{i\}$ denotes the structure sheaf shifted with respect to the $\k^\times$ action so that if $f \in \O_X(U)$ is a local function then viewed as a section $f' \in \O_X\{i\}(U)$ we have $t \cdot f' = t^{-i}(t \cdot f)$. We will denote by $\{i\}$ the operation of tensoring with $\O_X\{i\}$. In this paper we assume that any variety $X$ carries a $\k^\times$ action and we will denote by $D(X)$ the bounded derived category of $\k^\times$-equivariant coherent sheaves on $X$.

\subsubsection{Definition}

Once again we fix a base field $\k$. A {\bf geometric categorical $\sl_2$ action} consists of the following data.

\begin{enumerate}
\item A sequence of smooth varieties $Y(-N), Y(-N+1), \dots, Y(N-1), Y(N)$ over $\k$ (equipped with $\k^\times$-actions)
\item Fourier-Mukai kernels
$$\sE^{(r)}(\l) \in  D(Y(\l-r) \times Y(\l+r)) \text{ and } \sF^{(r)}(\l) \in  D(Y(\l+r) \times Y(\l-r)).$$
(which are $\k^\times$-equivariant). 
We will usually write $\sE(\l)$ for $\sE^{(1)}(\l)$ and $\sF(\l)$ for $\sF^{(1)}(\l)$ while one should think of $\sE^{(0)}(\l)$ and $\sF^{(0)}(\l)$ as $\O_\Delta$.
\item For each $Y(\l)$ a flat deformation $\tY(\l) \rightarrow \bA^1_\k$ carrying a $\k^\times$-action which maps fibres to fibres and acts on the base via $x \mapsto t^2 x$ (where $t \in \k^\times$). We call this a compatible $\k^\times$-action. 
\end{enumerate}

\begin{Remark}
Strictly speaking we only need a first order deformation $\tY(\lambda) \rightarrow \Spec(\k[x]/x^2)$ but in all our examples we obtain such a first order deformation from a natural deformation over $\bA^1_\k$. However, we could replace $\bA^1_\k$ by $\Spec(\k[x]/x^2)$ in the rest of paper and very little would change. 
\end{Remark}

On this data we impose the following additional conditions. We always work $\k^\times$ equivariantly. 

\begin{enumerate}
\item Each (graded piece of the) $\Hom$ space between two objects in $D(Y(\l))$ is finite dimensional. In particular, this means that $\End(\O_{Y(\l)}) = \k \cdot I$. 
\item 
$\sE^{(r)}(\l)$ and $\sF^{(r)}(\l)$ are left and right adjoints of each other up to shift. More precisely
\begin{enumerate}
\item $\sE^{(r)}(\l)_R = \sF^{(r)}(\l)[r\l]\{-r\l\}$
\item $\sE^{(r)}(\l)_L = \sF^{(r)}(\l)[-r\l]\{r\l\}$.
\end{enumerate}
\item 
At the level of cohomology of complexes we have
$$\H^*(\sE(\l+r) * \sE^{(r)}(\l-1)) \cong \sE^{(r+1)}(\l) \otimes_{\k} H^\star(\p^{r}).$$

\item 
If $ \l \le 0 $ then
\begin{equation*}
\sF(\l+1) * \sE(\l+1) \cong \sE(\l-1) * \sF(\l-1) \oplus \sP
\end{equation*}
where $\H^*(\sP) \cong \O_\Delta \otimes_\k H^\star(\p^{-\l-1})$. 

Similarly, if $\l \ge 0$ then
\begin{equation*}
\sE(\l-1) * \sF(\l-1) \cong \sF(\l+1) * \sE(\l+1) \oplus \sP'
\end{equation*}
where $\H^*(\sP') \cong \O_\Delta \otimes_\k H^\star(\p^{\l-1})$.

\item We have 
$$\H^*(i_{23*} \sE(\l+1) * i_{12*} \sE(\l-1)) \cong \sE^{(2)}(\l)[-1]\{1\} \oplus \sE^{(2)}(\l)[2]\{-3\}$$
where the $i_{12}$ and $i_{23}$ are the closed immersions
\begin{align*}
i_{12}: Y(\l-2) \times Y(\l) \rightarrow Y(\l-2) \times \tY(\l) \\
i_{23}: Y(\l) \times Y(\l+2) \rightarrow \tY(\l) \times Y(\l+2).
\end{align*}

\item If $\l \le 0$ and $k \ge 1$ then the image of $\supp(\sE^{(r)}(\l-r))$ under the projection to $Y(\l)$ is not contained in the image of $\supp(\sE^{(r+k)}(\l-r-k))$ also under the projection to $Y(\l)$. Similarly, if $\l \ge 0$ and $k \ge 1$ then the image of $\supp(\sE^{(r)}(\l+r))$ in $Y(\l)$ is not contained in the image of $\supp(\sE^{(r+k)}(\l+r+k))$. 
\item All $\sE^{(r)}$s and $\sF^{(r)}$s are sheaves (i.e. complexes supported in degree zero). 

\end{enumerate}

\begin{Remark}  \label{re:HT}
Having the conditions at the level of cohomology may seem strange but it is often much easier to prove that the cohomologies of two objects are the same than to prove that the objects are isomorphic. Moreover, there are natural examples where isomorphisms hold only at the level of cohomology. The moral is that also having deformations (with the properties described above) allows one to lift isomorphisms from the level of cohomology of objects to isomorphisms of objects. 

The idea of imposing the existence of a deformation was inspired by the work of Huybrechts-Thomas on $ \p^n $ objects.  In particular, in Proposition 1.4 of \cite{HT}, they show that under certain conditions, the deformation of a $ \p^n $ object is spherical.  When $ \l = -N+1 $, $ \sE(\l) $ is a relative $ \p^n $ object and we see (Proposition \ref{prop2:EF=FE}) that on the deformed varieties $i_* \sE(\l)$ satisfies a spherical-type condition. 

\end{Remark}

\subsection{The main result}
If we compare the definitions of strong categorical $ \sl_2$ action and geometric categorical $ \sl_2 $ action, we find the same functors (once we pass from $ \sE^{(r)}(\l)$ to $ \Phi_{\sE^{(r)}(\l)}$) which satisfy the same isomorphisms (compare points (iii), (iv) in the definitions above).  There are two main differences between the two definitions.  First, in the geometric version, the functors satisfy the isomorphisms only on the level of cohomology (an a priori weaker statement).  Second, in the geometric version, the isomorphisms are just abstract ismorphisms.  In the strong version, these isomorphisms come from specified morphisms, which themselves must satisfy nil affine Hecke relations.  However, despite these differences, the main result of this paper is that a geometric categorical $ \sl_2 $ gives a strong categorical $ \sl_2 $ action. 

\begin{Theorem}\label{th:main} Given a geometric categorical $\sl_2$ action set 
$$\D(\l) := D(Y(\l)) \text{ and } \E^{(r)}(\l) := \Phi_{\sE^{(r)}(\l)} \text{ and } \F^{(r)}(\l) := \Phi_{\sF^{(r)}(\l)}$$ 
where the shift in $\D(\l)$ is given by $\la r \ra = [r]\{-r\}$. Then there exist morphisms $\iota, \pi, \varepsilon, \eta, X, T$ giving a strong categorical $\sl_2$ action. Moreover, the choice of the $X$ and $T$ is parametrized by 
$$V(-1)^{tr} \times V(-2)^{tr} \times \k^\times \cong V(1)^{tr} \times V(2)^{tr} \times \k^\times$$
where $V(\l)^{tr} \subset \Ext^2(\O_\Delta, \O_\Delta)$ denotes the linear subspace of transient maps defined in section \ref{sec:reliii}. The choices of $\iota, \pi, \varepsilon$ and $\eta$ are unique up to scaling by $\k^\times$. 
\end{Theorem}

\begin{Remark} One may very well choose to ignore the $\k^\times$-action and nothing in the statement or proof of Theorem \ref{th:main} would change (except that we would have no $\{\cdot\}$ shift and $\la \cdot \ra = [\cdot]$). One reason to include the $\k^\times$ is because it occurs naturally in many of the examples we know and provides another grading which will be useful in future work. Also, there are examples (such as the one below) where the condition that $\End(\O_\Delta(\l))$ be finite dimensional fails if one doesn't work $\k^\times$-equivariantly. 
\end{Remark}

\section{The main example}\label{sec:mainexample}

Before proceeding with the proof of main Theorem \ref{th:main} we give an example of a geometric categorical $\sl_2$ action. We work over the base field $\k = \C$. The spaces involved will be cotangent bundles to Grassmannians. In \cite{ckl1} we gave an example of a geometric categorical $\sl_2$ action which is essentially a natural compactification of the one here. However, we prefer the one given here since in some ways it is simpler and more fundamental. 

\subsection{Spaces and functors}
Fix $ N > 0 $. For our spaces $Y(\l)$ we will take the total cotangent bundle to the Grassmannian $T^\star \bG(k,N)$ where $k = (N-\l)/2$. The $\C^\times$ will act naturally on the fibres of the bundle. These spaces have a particularly nice geometric description as 
$$T^\star \bG(k,N) \cong \{(X,V): X \in \End(\C^N), 0 \subset V \subset \C^N, \dim(V) = k \text{ and } \C^N \xrightarrow{X} V \xrightarrow{X} 0 \}$$
where $\End(\C^N)$ denotes the space of complex $N \times N$ matrices (the notation $ \C^N \xrightarrow{X} V \xrightarrow{X} 0 $ means that $ X(\C^n) \subset V $ and that $ X(V) = 0$). The action of $t \in \C^\times$ is by $X \mapsto t^2 X$. 

Forgetting $X$ corresponds to the projection $T^\star \bG(k,N) \rightarrow \bG(k,N)$ while forgetting $V$ gives a resolution of the variety 
$$ \{ X \in \End(\C^N) : X^2 =  0 \text{ and } \mbox{rank}(X) \le \min(k, N-k)\}. $$

On $Y(\l) = T^\star \bG(k,N)$ we have the tautological vector bundle $V$ as well as the quotient $\C^N/V$. 

To describe the kernels $\sE$ and $\sF$ we will need the correspondences 
$$ W^r(\l) \subset T^\star \bG(k+r/2, N) \times T^\star \bG(k-r/2, N) $$ 
defined by
\begin{align*}
W^r(\l) := \{ (X,V,V') : &X \in \End(\C^N), \dim(V) = k + \frac{r}{2}, \dim(V') = k - \frac{r}{2}, 0 \subset V' \subset V \subset \C^N  \\ 
& \C^N \xrightarrow{X} V' \text{ and } V \xrightarrow{X} 0 \}
\end{align*}
(here as before $ \lambda$ and $k $ are related by the equation $ k = (N - \l)/2 $).

There are two natural projections $\pi_1: (X,V,V') \mapsto (X,V)$ and $\pi_2: (X,V,V') \mapsto (X,V')$ from $W^r(\l)$ to $Y(\l-r)$ and $Y(\l+r)$ respectively. Together they give us an embedding
$$(\pi_1, \pi_2): W^r(\l) \subset Y(\l-r) \times Y(\l+r).$$
Notice that we also have a natural $\C^\times$ action on $W^r(\l)$ given by $X \mapsto t^2X$ so that both $\pi_1$ and $\pi_2$ are $\C^\times$-equivariant. 

On $W^r(\l)$ we have two natural tautological bundles, namely $V := \pi_1^*(V)$ and $V' := \pi_2^*(V)$ where the prime on the $V'$ indicates that the vector bundle is the pullback of the tautological bundle by the second projection. We also have natural inclusions 
$$0 \subset V' \subset V \subset \C^N \cong \O_{W^r(\l)}^{\oplus N}.$$

We now define the kernel $\sE^{(r)}(\l) \in  D(Y(\l-r) \times Y(\l+r))$ by 
\begin{equation*}
\sE^{(r)}(\l) := \O_{W^r(\l)} \otimes \det(\C^N/V')^{-r} \det(V)^r \{\frac{r(N-\l-r)}{2}\}.
\end{equation*}
Similarly, the kernel $\sF^{(r)}(\l) \in D(Y(\l+r) \times Y(\l-r))$ is defined by
\begin{equation*}
\sF^{(r)}(\l) := \O_{W^r(\l)} \otimes \det(V'/V)^{\l} \{\frac{r(N+\l-r)}{2}\}.
\end{equation*}
Notice that here $V' = \pi_2^*(V)$ is the pullback from the projection onto $Y(\l-r)$ since we now view $Y(\l-r)$ as being the second factor rather than the first. 

\begin{Remark}
Although $W^r(\l)$ is not proper the projections onto $Y(\l-r)$ and $Y(\l+r)$ are proper since the fibres are subvarieties of Grassmannians. Hence $\sE^{(r)}(\l)$ and $\sF^{(r)}(\l)$ induce FM transforms $\E^{(r)}(\l)$ and $\F^{(r)}(\l)$ between $D(Y(\l-r))$ and $D(Y(\l+r))$. 
\end{Remark}

\subsection{Deformations}

$Y(\l) = T^\star \bG(k,N)$ has a natural 2-parameter deformation over $\bA_\C^2 $, whose fibre at the point $ (x,y) $ is  given by
$$\{ (X,V): X \in \End(\C^N), 0 \subset V \subset \C^N, \dim(V) = k \text{ and } X|_{V} = x \cdot I, X|_{\C^N/V} = y \cdot I \}.$$
Notice that the fibre over $(x,y)=(0,0)$ recovers $T^\star \bG(k,N)$. This deformation restricted to the diagonal $x=y$ is actually trivial but if we take any other ray in $\bA_\C^2$ through the origin we get a non-trivial deformation of $T^\star \bG(k,N)$. Which ray we choose is not that important -- we choose the axis $y=0$ to obtain the deformation
$$\tY(\l) = \{ (X,V, x): x \in \C, X \in \End(\C^N), 0 \subset V \subset \C^N, \dim(V) = k \text{ and } X|_{V} = x \cdot I, X|_{\C^N/V} = 0 \}$$
where $\l = N-2k$. The $\C^\times$-action here acts, like before, by $X \mapsto t^2X$ and by $x \mapsto t^2x$. Thus we have a compatible $\C^\times$-action. 

\subsection{A geometric categorical $\sl_2$ action}

\begin{Theorem} \label{th:geomsl2}
The varieties $Y(\l) = T^* \bG(k,n) $ along with the deformations $ \tY(\l) $ and kernels $ \sE^{(r)}(\l)$, $\sF^{(r)}(\l) $, give a geometric categorical $\sl_2 $ action.
\end{Theorem}

Recall that in \cite{ckl1} we constructed a geometric categorical $\sl_2$ action on certain spaces $Y(k,l)$. These varieties are compactifications of $T^* \bG(k,N)$ when $k+l=N$. One way to prove Theorem \ref{th:geomsl2} is to repeat the proof in \cite{ckl1} word by word replacing $Y(k,l)$ by $T^* \bG(k,N)$ at each step. The geometry is the same since we just restrict to open subsets. 

Alternatively, one can show that the geometric categorical $\sl_2$ action from \cite{ckl1} formally implies Theorem \ref{th:geomsl2}. We choose this approach because repeating the argument in \cite{ckl1} is a bit tedious and repetitive. We think it is more instructive to spell out the relationship between these two categorical $\sl_2$ actions and see how the one here follows directly from \cite{ckl1}. 

\subsection{Relation to categorification of skew Howe duality}

In \cite{ckl1} we constructed a geometric categorical $\sl_2$ action on varieties which compactified the above cotangent bundles. We now explain how that categorification is related to the one above. 

In \cite{ckl1} we fixed integers $m,N$ and defined varieties $Y(k,l)$ and functors $\sE^{(r)}(k,l)$ where $k+l=N$. However, only the case when $m=N$ is related to the example above. We now recall these varieties and functors when $m=N$. 

We define
\begin{equation*}
Y(k,l) := \{ \C^N \otimes \C[[z]] = L_0 \subset L_1 \subset L_2 \subset \C^N \otimes \C((z)) : z L_i \subset L_{i-1}, \dim(L_1/L_0) = k, \dim(L_2/L_1) = l \}
\end{equation*}
so that $D(Y(k,l))$ will correspond to the weight space of weight $\l = l-k$. The $\C^\times$-action on $Y(k,l)$ is induced by $t \cdot z^k = t^{2k}z^k$ acting on $\C((z))$. 

Next we define correspondences 
\begin{equation*}
W^r(k,l) := \{ (L_\bullet, L'_\bullet) : L_1 \subset L'_1, L_2 = L'_2 \} \subset Y(k,l) \times Y(k+r,l-r)
\end{equation*}
followed by kernels
\begin{equation*}
\sE^{(r)}(k,l) := \O_{W^r(k,l)} \otimes \det(L_2/L'_1)^{-r} \det(L_1/L_0)^r \{rk\} \in D(Y(k+r,l-r) \times Y(k,l))
\end{equation*}
and
\begin{equation*}
\sF^{(r)}(k,l) := \O_{W^r(k,l)} \otimes \det(L'_1/L_1)^{l-r-k} \{r(l-r)\} \in D(Y(k,l) \times Y(k+r,l-r))
\end{equation*}
where (abusing notation a little) $L_i$ denotes the tautological bundle on $Y(k,l)$ whose fibre over the point $(L_0 \subset L_1 \subset L_2)$ is $L_i$. As usual, one can check that everything here is $\C^\times$-equivariant. 

Since $z^2 L_2 \subset L_0$ this means that $\C^N \otimes \C[[z]] \subset L_2 \subset \C^N \otimes z^{-2} \C[[z]]$.  Hence the $ N $ dimensional space $ L_2/L_0 $ is a subspace of the $ 2N $ dimensional vector space $\C^N \otimes z^{-2}\C[[z]]/C[[z]]$. Define the linear projection $P: \C^N \otimes z^{-2}\C[[z]]/C[[z]] \rightarrow \C^N $ by $P(v \otimes z^{-1}) = v $ and $P(v \otimes z^{-2}) = 0$.  Now consider the open subvariety
$$U(k,l) := \{L_\bullet: P(L_2) = \C^N \} \subset Y(k,l).$$

The following result is due to \cite[Theorem 5.3]{MVy} and \cite[Lemma 2.3.1]{Ngo}
\begin{Lemma}\label{lem:U(k,l)} $U(k,l) \cong T^\star \bG(k,N)$ via the isomorphism 
$$(L_0 \subset L_1 \subset L_2) \mapsto (P|_{L_2} \circ z \circ P^{-1}|_{L_2}, P(L_1/L_0)) = (X,V).$$
Moreover, this isomorphism is $\C^\times$-equivariant with respect to the $\C^\times$-actions on $Y(k,l)$ and $T^\star \bG(k,N)$ defined above. 
\end{Lemma}
\begin{proof}
We give a sketch. The definition of $U(k,l)$ implies that if $ L_\bullet \in U(k,l) $, then $ P $ gives an isomorphism between $ L_2 $ and $ \C^N $.  So $P$ takes $L_1$ to a $k$-dimensional subspace $V \subset \C^N$ and $0 \xleftarrow{z} L_1 \xleftarrow{z} L_2$ induces the map $0 \xleftarrow{X} V \xleftarrow{X} \C^N$ where $X = P|_{L_2} \circ z \circ P^{-1}|_{L_2}$. The fact that this isomorphism is $\C^\times$-equivariant follows since the $\C^\times$-actions are given by $X \mapsto t^2 X$ and $z \mapsto t^2 z$. 
\end{proof}

Now in \cite{ckl1} we also had deformations 
$$\tY(k,l) := \{ L_\bullet: z|_{L_1/L_0} = x \cdot I, z|_{L_2/L_1} = 0, \dim(L_1/L_0) = k, \dim(L_2/L_1) = l \}.$$
These were also equipped with $\C^\times$-actions induced by $t \cdot z^k = t^{2k} z^k$ and $t \cdot x = t^2 x$. From \cite[Theorem 5.3]{MVy}, we obtain the following result:

\begin{Lemma}
The embedding $ T^\star \bG(k,N) \cong U(k,l) \rightarrow Y(k,l) $ of Lemma \ref{lem:U(k,l)} extends to an embedding $ \tY(\l) \rightarrow \tY(k,l) $ which is compatible with the projections to $ \bA_\C^1 $ and the $\C^\times$-actions given above.
\end{Lemma}

\subsection{Proof of Theorem \ref{th:geomsl2}}

The idea is to show that the categorical $\sl_2$ relations for the $Y(k,l)$ varieties induce the same relations for our open subvarieties $T^\star \bG(k,N)$. 

To do this we begin with the following observation. Denote by $j: Y(\l) = T^\star \bG(k,N) \rightarrow Y(k,l)$ the natural open immersion. Then it is not difficult to see that 
$$(j \times j)^{-1} W^r(k,l) = W^r(\l) \subset T^\star \bG(k,N) \times T^\star \bG(k+r,N)$$
where $\l = N-2k-r$. Even better, we have 
$$ (j \times 1)^{-1} W^r(k,l) = W^r(\l) \subset T^\star \bG(k,N) \times Y(k+r, l-r) $$

Moreover, since $L_1$ and $L_1'$ on $W^r(k,l)$ correpond to $V$ and $V'$ on $W^r(\l)$ it is easy to check that
$$(j \times j)^* \sE^{(r)}(k,l) \cong \sE^{(r)}(\l) \text{ and } (j \times j)^* \sF^{(r)}(k,l) = \sF^{(r)}(\l)$$
and that
\begin{equation} \label{eq:j11j}
 (j \times 1)^* \sE^{(r)}(k,l) \cong (1 \times j)_* \sE^{(r)}(\l) \text{ and } (j \times 1)^* \sF^{(r)}(k,l) = (1 \times j)_* \sF^{(r)}(\l).
\end{equation}

We can now make use of the following lemma.
\begin{Lemma}\label{lem:inducerelation}
Suppose that $ Y_1, Y_2, Y_3 $ are smooth varieties and $ U_1, U_2, U_3 $ are open subvarieties.  Let $ j_a : U_a \rightarrow Y_a $ denote the natural open immersion.  Let $ F_{12} \in D(Y_1 \times Y_2), F_{23} \in D(Y_2 \times Y_3) $ denote objects on the products and let 
$$ F_{12}' := (j_1 \times j_2)^*(F_{12}) \in D(U_1 \times U_2) \text{ and } F_{23}' := (j_2 \times j_3)^*(F_{23}) \in D(U_2 \times U_3).$$ 
Suppose moreover that 
\begin{equation*}
 (j_1 \times 1)^* F_{12} \cong (1 \times j_2)_* F_{12}' \in D(U_1 \times Y_2) \text{ and } (1 \times j_3)^* F_{23} = (j_2 \times 1)_* F_{23}' \in D(Y_2 \times U_3).
\end{equation*}

Then $ F'_{23} * F'_{12} \cong (j_1 \times j_3)^*(F_{23} * F_{12}) $
\end{Lemma}
\begin{proof}
This follows by a direct calculation. We have
\begin{eqnarray*}
(j_1 \times j_3)^* (F_{23} * F_{12}) 
&\cong& (j_1 \times j_3)^* \pi_{13*} (\pi_{12}^* F_{12} \otimes \pi_{23}^* F_{23}) \\
&\cong& p_{13*} (j_1 \times 1 \times j_3)^* (\pi_{12}^* F_{12} \otimes \pi_{23}^* F_{23}) \\
&\cong& p_{13*} \left( (j_1 \times 1 \times j_3)^* \pi_{12}^* F_{12} \otimes (j_1 \times 1 \times j_3)^* \pi_{23}^* F_{23} \right) \\
&\cong& p_{13*} \left( p_{12}^* (j_1 \times 1)^* F_{12} \otimes p_{23}^* (1 \times j_3)^* F_{23} \right) \\
&\cong& p_{13*} \left( p_{12}^* (1 \times j_2)_* F_{12}' \otimes p_{23}^* (j_2 \times 1)_* F_{23}' \right) \\
&\cong& p_{13*} \left( (1 \times j_2 \times 1)_* {\pi'_{12}}^* F_{12}' \otimes (1 \times j_2 \times 1)_* {\pi'_{23}}^* F_{23}' \right) \\
&\cong& p_{13*} (1 \times j_2 \times 1)_* \left( {\pi'_{12}}^* F_{12}' \otimes {\pi'_{23}}^* F_{23}' \right) \\
&\cong& \pi'_{13*} ({\pi'_{12}}^* F_{12}' \otimes {\pi'_{23}}^* F_{23}' ) \\
&\cong& F_{23}' * F_{12}'
\end{eqnarray*}
where $p_{ab}$ is the projection from $U_1 \times Y_2 \times U_3$ onto the $a,b$ factor and ${\pi'_{ab}}$ is the projection from $U_1 \times U_2 \times U_3$ onto the $a,b$ factor. To get the 2nd and 6th isomorphisms we used commutativity of pushing and pulling in a flat base change. To get the 7th isomorphism we used that $1 \times j_2 \times 1$ is an open immersion so tensoring commutes with pushforward. 
\end{proof}

Now we can finish the proof of Theorem \ref{th:geomsl2}. We need to check that the $\sE(\l)$'s and $\sF(\l)$'s satisfy conditions (i) - (vii) for having a geometric categorical $\sl_2$ action. 
Conditions (ii) and (iii) follow immediately since we are just restricting to open subsets. Condition (vii) can be checked quite easily just like in \cite{ckl1} by noting that (if $\l \le 0$ and $k \ge 1$) then the image of $\supp(\sE^{(r)}(\l-r))$ contains points where the kernel of $X$ has dimension $r+(N+\l)/2$ while the image of $\supp(\sE^{(r)}(\l-r-k))$ is contained in the locus where the kernel of $X$ has dimension $\ge r+k+(N+\l)/2)$ -- so the former cannot be contained in the latter.

Next we check condition (iv). We apply Lemma \ref{lem:inducerelation} with $ U_1 := Y(\l-1-r), U_2 := Y(\l-1+r), U_3 := Y(\l+1+r), Y_1 := Y(k+r,l-r), Y_2 := Y(k,l), Y_3 := Y(k-1,l+1) $, where $\l = N-2k-r+1$. We choose $ F_{12} := \sE^{(r)}(k,l), F_{23} := \sE(k-1,l+1) $.  The main hypothesis of Lemma \ref{lem:inducerelation} follows from (\ref{eq:j11j}). From the conclusion of Lemma \ref{lem:inducerelation}, we deduce that 
\begin{equation*}
\sE(\l+r) * \sE^{(r)}(\l-1)  \cong (j_1 \times j_3)^*(\sE(k-1, l+1) * \sE^{(r)}(k,l)).
\end{equation*}
Applying $ \H^* $ to both sides, and using the fact that the underived pullback $(j_1 \times j_3)^*$ is exact, we obtain 
\begin{eqnarray*}
 \H^*(\sE(\l+r) * \sE^{(r)}(\l-1)) 
&\cong& (j_1 \times j_3)^*(\H^*(\sE(k-1, l+1) * \sE^{(r)}(k,l))) \\
&\cong& (j_1 \times j_3)^*(\sE^{(r+1)}(k-1,l+1) \otimes_\C H^\star(\p^r)) \\
&\cong& \sE^{(r+1)}(\l) \otimes_\C H^\star(\p^r).
 \end{eqnarray*}
Thus the first part of relation (iv) follows from the corresponding relation (iv) for the composition of the $\sE(k,l)$'s.  

To prove condition (v) we apply the Lemma with $ Y_1, Y_3, U_1, U_3 $ as above but with $ Y_2 = \tY(k,l) $ and $ U_2 = \tY(\l) $. The relation then follows as above. Condition (v) also follows by a similar argument. 

The final thing to check is condition (i): namely that $\Hom^i(A,B)$ is finite dimensional for any $A,B \in Y(\l)$. Since $Hom^i(A,B) \cong H^i(A^\vee \otimes B)$ it suffices to check $H^i(A)$ is finite dimensinoal for any $A$. By considering the corresponding spectral sequence we can even assume $A$ is a sheaf. 

Now let $\pi$ be the projection $T^* \bG(k,N) \rightarrow \bG(k,N)$. The fibres are $\bA^{k(N-k)}$ so that $\pi_*(A) = R^0 \pi_*(A)$ ({\it i.e.} there is no higher cohomology). Now the $\C^\times$ action acts naturally on the fibres. Since $H^0(\bA^{k(N-k)}, M)^{\C^\times}$ is finite dimensional for any $\C^\times$-equivariant coherent sheaf $M$ we have that $R^0 \pi_*(A)^{\C^\times}$ is a coherent sheaf. Hence $H^i(A) = H^i(R^0 \pi_*(A)^{\C^\times})$, is finite dimensional since $\bG(k,N)$ is proper. 

Notice that (i) would not hold if we were to ignore the $\C^\times$-action. 

\begin{Remark}
It follows immediately from Theorem \ref{th:geomsl2} that $U_q(\sl_2)$ acts on the Grothendieck group
$$K(N):=\bigoplus_{\l=-N}^{N} K(\D(\l))$$
where $ K(\D(\l)) $ is a $\C[q,q^{-1}] $ module with $ -q $ acting by $ \{1\}$.

The weight space $ K(\D(\l)) $ has dimension $ \dim H^*(T^\star(\bG(k,n)))) = \binom{N}{k} $ by the argument in Proposition 7.2 of \cite{ck2}.  Hence as a $ U_q(\sl_2) $ representation, $K(N)$ is isomorphic to the $N$th tensor power of the irreducible 2-dimensional representation.
\end{Remark}

\section{Obtaining formality from deformations $\tY(\l)$} \label{se:formality}

The rest of the paper is devoted to the proof of Theorem \ref{th:main}. We will assume throughout that we have a fixed geometric categorical $\k^\times$-equivariant $\sl_2$ action.

The most difficult part of the proof (by far) is to construct the $X(\l)$'s and $T(\l)$'s so that they satisfy the nil affine Hecke relations. So we first check all the other properties and leave the nil affine Hecke relations until the end. 

\subsection{Some deformation theory}

We begin with some general deformation theory. Our deformations will be over $\bA^1_\k$ although the arguments below are valid over more general one-dimensional bases. 

Suppose $\tY \rightarrow \bA^1_\k$ is a flat deformation of a variety $Y$ and denote by $i: Y \rightarrow \tY$ the inclusion of $Y$ as the fibre over $0 \in \bA^1_\k$. We assume, as above, that we have a compatible $\k^\times$-action on $\tY$ (i.e. it maps fibres to fibres and acts on the base $\bA^1_\k$ by $x \mapsto t^2 x$).  

Given an object $\sG \in D(Y)$  we get the standard exact triangle
$$\sG[1] \otimes N_{Y/\tY}^\vee \rightarrow i^* i_* \sG \rightarrow \sG$$
obtained via the natural adjunction maps. Now $N_{Y/\tY} \cong \O_Y \{2\}$ the connecting morphism gives an endomorphism $\alpha: \sG[-1]\{1\} \rightarrow \sG[1]\{-1\}$. Alternatively, $\alpha$ can be defined as the composition $\At(\sG) \cdot \kappa(i) \in \Ext^2(\sG,\sG)$ where $\At(\sG) \in \Ext^1(\sG,\sG \otimes \Omega_Y)$ is the Atiyah class of $\sG$ and $\kappa(i) \in \Ext^1(\Omega_Y, N_{Y/\tY})$ is the Kodaira-Spencer class (see the Appendix of \cite{HT} for a proof of the equivalence of these two definitions). From either definition it is apparent that $i: Y \rightarrow \tY$ only defines the map $\alpha$ up to a non-zero multiple because, though $N_{Y/\tY} \cong \O_Y \{2\}$, this isomorphism is not canonical. Nevertheless, regardless of the value of this non-zero multiple we will always have 
$$\Cone(\sG[-1]\{1\} \xrightarrow{\alpha} \sG[1]\{-1\}) \cong i^* i_* \sG \{1\}.$$

\begin{Lemma}\label{lem:j1} Given three spaces $Y_i$ ($i=1,2,3$) and deformation $\tY_2 \rightarrow \bA^1_\k$ of $Y_2$ denote by $i_{12}: Y_1 \times Y_2 \rightarrow Y_1 \times \tY_2$ and $i_{23}: Y_2 \times Y_3 \rightarrow \tY_2 \times Y_3$ the natural inclusions. Given $\sG_{12} \in D(Y_1 \times Y_2)$ and $\sG_{23} \in D(Y_2 \times Y_3)$ we have 
\begin{eqnarray*}
\sG_{23} * (i_{12}^* i_{12*} \sG_{12}) \cong (i_{23*} \sG_{23}) * (i_{12*} \sG_{12}) \cong (i_{23}^* i_{23*} \sG_{23}) * \sG_{12} 
\end{eqnarray*}
Everything still holds if we also have compatible $\k^\times$-actions. 
\end{Lemma}
\begin{proof}
The proof amounts to diagram chasing.
\begin{align*}
\sG_{23} * i_{12}^* i_{12*} \sG_{12} &\cong \pi_{13*} \left( \pi_{12}^* (i_{12}^* i_{12*} \sG_{12}) \otimes \pi_{23}^* \sG_{23} \right) \\
&\cong \tilde{\pi}_{13*} \tilde{i}_* \left( \tilde{i}^* \tilde{\pi}_{12}^* i_{12*} \sG_{12} \otimes \pi_{23}^* \sG_{23} \right) 
\end{align*}
where we use the commuting squares
\begin{equation*}
\xymatrix{
Y_1 \times Y_2 \times Y_3 \ar[r]^{{\tilde{i}}} \ar[d]_{\pi_{12}} & Y_1 \times \tY_2 \times Y_3 \ar[d]_{{\tilde{\pi}}_{12}} \\
Y_1 \times Y_2 \ar[r]^{i_{12}} & Y_1 \times \tY_2 
} \quad \quad
\xymatrix{
Y_1 \times Y_2 \times Y_3 \ar[r]^{{\tilde{i}}} \ar[d]_{\pi_{13}} & Y_1 \times \tY_2 \times Y_3 \ar[d]_{{\tilde{\pi}}_{13}} \\
Y_1 \times Y_3 \ar[r]^{\mbox{id}} & Y_1 \times Y_3
}
\end{equation*}
The projection formula gives
\begin{eqnarray*}
\tilde{\pi}_{13*} \tilde{i}_{*} \left( \tilde{i}^* \tilde{\pi}_{12}^* i_{12*} \sG_{12} \otimes \pi_{23}^* \sG_{23} \right) 
&\cong& \tilde{\pi}_{13*} \left( \tilde{\pi}_{12}^* (i_{12*} \sG_{12}) \otimes \tilde{i}_{*} \pi_{23}^* \sG_{23} \right) \\ 
&\cong& \tilde{\pi}_{13*} \left( \tilde{\pi}_{12}^* (i_{12*} \sG_{12}) \otimes \tilde{\pi}_{23}^* (i_{23*} \sG_{23}) \right) \\
&\cong& i_{23*} \sG_{23} * i_{12*} \sG_{12}
\end{eqnarray*}
where the second line follows from flat base change on the commuting square
\begin{equation*}
\xymatrix{
Y_1 \times Y_2 \times Y_3 \ar[r]^{\tilde{i}} \ar[d]_{\pi_{23}} & Y_1 \times \tY_2 \times Y_3 \ar[d]_{{\tilde{\pi}}_{23}} \\
Y_2 \times Y_3 \ar[r]^{i_{23}} & \tY_2 \times Y_3.
}
\end{equation*}
This proves the first isomorphism in the Lemma. The second isomorphism follows similarly. If we also have a compatible $\k^\times$-action nothing changes in the proof since all the maps are naturally $\k^\times$-equivariant. 
\end{proof}

\subsection{Formality of $\E^{(r_2)} \circ \E^{(r_1)} \cong \E^{(r_1+r_2)} \otimes H^\star(\bG(r_1,r_1+r_2))$}

\begin{Proposition}\label{prop:theta_k}
We have the direct sum decomposition
$$\sE * \sE^{(r)} \cong \sE^{(r+1)} \otimes_\k H^*(\p^r) \cong \sE^{(r)} * \sE.$$
In the deformed setup, at the level of cohomology, we have
$$\H^*(i_{23*} \sE * i_{12*} \sE^{(r)}) \cong \sE^{(r+1)}[-r]\{r\} \oplus \sE^{(r+1)}[r+1]\{-r-2\}$$ 
where $i_{12}$ and $i_{23}$ are the inclusions
\begin{align*}
i_{12}: Y(\l-2r) \times Y(\l) \rightarrow Y(\l-2r) \times \tY(\l) \\
i_{23}: Y(\l) \times Y(\l+2) \rightarrow \tY(\l) \times Y(\l+2). 
\end{align*}
We also get for free the same relations if we replace all the $\sE$s above by $\sF$s.
\end{Proposition}
\begin{proof}
By Lemma \ref{lem:j1} we have
$$i_{23*} \sE * i_{12*} \sE^{(r)} \cong i_{23}^* i_{23*} \sE * \sE^{(r)}.$$
Using the standard exact triangle $\sE [1]\{-2\} \rightarrow i_{23}^* i_{23*} \sE \rightarrow \sE$ we find that
\begin{eqnarray}\label{eq:16}
i_{23*} \sE * i_{12*} \sE^{(r)} 
&\cong& \Cone (\sE[-1] \xrightarrow{\gamma} \sE[1]\{-2\}) * \sE^{(r)} \cong \Cone (\sE * \sE^{(r)} [-1] \xrightarrow{\gamma I} \sE * \sE^{(r)} [1]\{-2\})
\end{eqnarray}
where $\gamma$ is the connecting map in the standard triangle above. Basically, we need to understand the map induced by $\gamma I$ at the level of cohomology.

When $r = 1$ we know (by assumption) that the left side of equation (\ref{eq:16}) is isomorphic (at the level of cohomology) to $\sE^{(2)}[-1]\{1\} \oplus \sE^{(2)}[2]\{-3\}$ so that $\gamma I$ must induce an isomorphism (at the level of cohomology) on one summand $\sE^{(2)}$ (keep in mind that all $\sE^{(r)}$ are sheaves). Now consider the map
$$\sE^{(2)}[1]\{-1\} \oplus \sE^{(2)}[-1]\{1\} \xrightarrow{\iota \oplus (\gamma I) \circ \iota} \sE * \sE$$
where $\iota$ is the inclusion of $\sE^{(2)}[1]\{-1\}$ into the lowest degree cohomology of $\sE * \sE$. Note that we do not need to know $\sE * \sE$ is formal in order to define $\iota$ (in general, if you have a complex $C^\cdot$ bounded from below then you can include its lowest non-zero cohomology into it $\H^{l}(C^\cdot) \rightarrow C^\cdot$). By the fact above (about $\gamma I$) this map induces an isomorphism at the level of cohomology and so is an isomorphism $\sE * \sE \cong \sE^{(2)} \otimes_\k H^*(\p^1)$ (quasi-isomorphisms are by definition isomorphisms in the derived category).

This completes the base case $r=1$. We now proceed by induction on $r$. First we need to understand the map
$$\sE * \sE^{(r)} [-1]\{1\} \xrightarrow{\gamma I} \sE * \sE^{(r)} [1]\{-1\}$$
at the level of cohomology. We know $\H^*(\sE * \sE^{(r)}) \cong \sE^{(r+1)} \otimes_\k H^*(\p^r)$ and we need to show $\gamma I$ induces an isomorphism between all but a pair of summands $\sE^{(r+1)}$ (the one in highest cohomological degree on the left side and lowest cohomological degree on the right side).

Suppose this were not the case. Then $\H^*(i_{23}^* i_{23*} \sE * \sE^{(r)}) \cong \Cone(\gamma I)$ would contain at least four summands $\sE^{(r+1)}$. By induction we know $\sE^{(r-1)} * \sE \cong \sE^{(r)} \otimes_\k H^*(\p^{r-1})$ so this would mean that
$$\H^*(i_{23}^* i_{23*} \sE * \sE^{(r-1)} * \sE)$$
contains at least $4r$ summands $\sE^{(r+1)}$. On the other hand, also by induction we know that $\H^*(i_{23}^* i_{23*} \sE * \sE^{(r-1)}) \cong \sE^{(r)} [-r+1]\{r-1\} \oplus \sE^{(r)}[r]\{-r-1\}$ which means we have an exact triangle
$$\sE^{(r)}[r]\{-r-1\} \rightarrow i_{23}^* i_{23*} \sE * \sE^{(r-1)} \rightarrow \sE^{(r)} [-r+1]\{r-1\}.$$
So $\H^*(i_{23}^* i_{23*} \sE * \sE^{(r-1)} * \sE)$ contains at most $2(r+1)$ summands $\sE^{(r+1)}$. Since $4r > 2(r+1)$ if $r > 1$ this is a contradiction.

Finally, as in the case $r=1$, we have a map
$$\sE^{(r+1)} \otimes_\k H^*(\p^r) \xrightarrow{\iota \oplus (\gamma I) \circ \iota \oplus \dots \oplus (\gamma I)^r \iota} \sE * \sE^{(r)}$$
where $\iota$ is the inclusion of $\sE^{(r+1)} [r]\{-r\}$ into the lowest cohomological degree of $\sE * \sE^{(r)}$ (note that, as before, we do not need to know that $\sE * \sE^{(r)}$ is formal in order to define $\iota$). This induces an isomorphism on cohomology (and hence must be an isomorphism $\sE^{(r+1)} \otimes_\k H^\star(\p^r) \xrightarrow{\sim} \sE * \sE^{(r)}$ in the derived category). 
\end{proof}

\begin{Remark} The fact that $\Ext^i(\sE^{(r+1)}(\l),\sE^{(r+1)}(\l)\{j\}) = 0$ for $i < 0$ (and any $j \in \Z$) while $\End(\sE^{(r+1)}(\l)) \cong \k \cdot I$ (see Lemma \ref{lem:homs2}) means that $\iota$ is actually unique (up to a non-zero multiple). Similarly we can define
$$\pi: \sE(\l+r) * \sE^{(r)}(\l-1) \rightarrow \sE^{(r+1)}(\l)[-r]\{r\}$$
as the natural projection out of the top ($\pi$ is likewise unique, up to non-zero multiple). 
\end{Remark}

\subsection{Formality of $\F \circ \E \cong \E \circ \F \oplus \id \otimes H^\star(\p)$}

The proof of formality here is analogous to the one in the last section. Recall that if $\l \le 0$ then $\sF(\l+1) * \sE(\l+1) \cong \sE(\l-1) * \sF(\l-1) \oplus \sP$. Lemma \ref{lem:FEtoP} below implies that given a map $\sF(\l+1) * \sE(\l+1) \rightarrow \sF(\l+1) * \sE(\l+1) [i]\{j\}$ there is an induced map $\sP \rightarrow \sP [i]\{j\}$ well defined up to a non-zero multiple. 

\begin{Lemma}\label{lem:FEtoP} If $\l \le 0$ we have 
$$\Hom(\sP, \sE(\l-1) * \sF(\l-1)) = 0 \text{ and } \Hom(\sE(\l-1) * \sF(\l-1), \sP) = 0$$
where $\H^*(\sP) \cong \O_\Delta \otimes_{\k} H^\star(\p^{-\l-1})$. 
\end{Lemma}
\begin{proof}
We have
\begin{eqnarray*}
\Hom(\O_\Delta[n]\{-n\}, \sE(\l-1) * \sF(\l-1)) 
&\cong& \Hom(\sE(\l-1)_L [n]\{-n\}, \sF(\l-1)) \\
&\cong& \Hom(\sF(\l-1) [-\l+1+n]\{\l-1-n\}, \sF(\l-1)) 
\end{eqnarray*}
which is zero if $n > \l - 1$. Since 
$$\H^*(\sP) \cong \bigoplus_{j=0}^{-\l-1} \O_\Delta[-\l-1-2j]\{\l+1+2j\}$$
this means $\Hom(\sP, \sE(\l-1) * \sF(\l-1)) = 0$. 

Similarly, we have 
\begin{eqnarray*}
\Hom(\sE(\l-1) * \sF(\l-1), \O_\Delta [n]\{-n\}) 
&\cong& \Hom(\sF(\l-1), \sE(\l-1)_R [n]\{-n\}) \\
&\cong& \Hom(\sF(\l-1), \sF(\l-1)[\l-1+n]\{-\l+1-n\})
\end{eqnarray*}
which is zero if $n < -\l+1$. This means $\Hom(\sE(\l-1) * \sF(\l-1), \sP) = 0$. 
\end{proof}

\begin{Proposition}\label{prop2:EF=FE}
If $\l \le 0$ we have
$$\sF(\l+1) * \sE(\l+1) \cong \sE(\l-1) * \sF(\l-1) \oplus \O_\Delta \otimes_{\k} H^\star(\p^{-\l-1}).$$
while if $\l \ge 0$ we have
$$\sE(\l-1) * \sF(\l-1) \cong \sF(\l+1) * \sE(\l+1) \oplus \O_\Delta \otimes_{\k} H^\star(\p^{\l-1}).$$

In the deformed setup, if $\l \le -1$ and we restrict away from $\mbox{supp}(\sE(\l-1) * \sF(\l-1))$, we have
$$\H^*(i_{23*} \sF(\l+1) * i_{12*} \sE(\l+1)) \cong \O_\Delta[-\l]\{\l-1\} \oplus \O_\Delta[\l+1]\{-\l-1\}$$
where $i_{12}$ and $i_{23}$ are the inclusions
\begin{align*}
i_{12}: Y(\l) \times Y(\l+2) \rightarrow Y(\l) \times \tY(\l+2) \\
i_{23}: Y(\l+2) \times Y(\l) \rightarrow \tY(\l+2) \times Y(\l)
\end{align*}
(and similarly if $\l \ge 1$).
\end{Proposition}
\begin{proof}
We suppose $\l \le 0$ (the other case is proved in the same way).  We assume $\l \le -2$ (if $\l = 0,-1$ there is nothing really to prove). 

The proof is similar to that of Proposition \ref{prop:theta_k}. Instead of studying equation (\ref{eq:16}) we look at
\begin{eqnarray}\label{eq:17}
\sF(\l+1) * i^*_{12} i_{12*} \sE (\l+1) \cong \Cone( \sF * \sE [-1] \xrightarrow{I \gamma} \sF * \sE [1]\{-2\}).
\end{eqnarray}
Now we know $\sF(\l+1) * \sE(\l+1) \cong \sE(\l-1) * \sF(\l-1) \oplus \sP$ and we want to show that the induced map $\H^*(\sP [-1]) \xrightarrow{I \gamma} \H^*(\sP[1]\{-2\})$ is an isomorphism on all but one pair of summands $\O_\Delta$ (the highest degree summand  on the left side and lowest degree summand on the right side). 

Suppose this is not the case. Then $\H^*(\sF(\l+1) * i^*_{12} i_{12*} \sE(\l+1))$ would contain at least four summands $\O_\Delta$. Now consider
$$\sF * \sE * i^*_{12} i_{12*} \sE (\l+1) \cong \Cone(\sF * \sE * \sE (\l+1) [-1] \xrightarrow{II \gamma} \sF * \sE * \sE (\l+1) [1]\{-2\}).$$
On the one hand the map $II \gamma$ can be rewritten as
$$\sE * \sF * \sE (\l+1) [-1] \oplus \sP' * \sE (\l+1)  [-1] \xrightarrow{II \gamma \oplus I \gamma} \sE * \sF * \sE (\l+1) [1]\{-2\} \oplus \sP' * \sE (\l+1)  [1]\{-2\}$$
where $\H^*(\sP') \cong \O_\Delta \otimes_\k H^\star(\p^{-\l-3})$. The induced map $\sP' * \sE [-1] \xrightarrow{I \gamma} \sP' * \sE [1]\{-2\}$ at the level of cohomology is zero. So when we take the cone we obtain $2(-\l-2)$ summands $\sE$. Meanwhile, by the assumption above,
$$\H^*(\Cone(\sE * \sF * \sE (\l+1) [-1]  \xrightarrow{II \gamma} \sE * \sF * \sE (\l+1) [1]\{-2\}))$$
contains at least $4$ summands $\sE$. So in total we have at least $-2\l$ summands $\sE$.

On the other hand, $\sF * \sE * \sE \cong \sF * \sE^{(2)} [-1]\{1\} \oplus \sF * \sE^{(2)} [1]\{-1\}$ and by Proposition \ref{prop:theta_k} the induced map
$$\sF * \sE^{(2)} [-2]\{2\} \oplus \sF * \sE^{(2)} \xrightarrow{II \gamma} \sF * \sE^{(2)} \oplus \sF * \sE^{(2)} [2]\{-2\} $$
induces an isomorphism between the summands $\sF * \sE^{(2)}$ on either side. But
\begin{eqnarray*}
\H^*(\sF * \sE * \sE (\l+1))
&\cong& \H^*(\sE * \sF * \sE) \oplus \sE \otimes_\k H^*(\p^{-\l-3}) \\
&\cong& \H^*(\sE * \sE * \sF) \oplus \sE \otimes_\k (H^*(\p^{-\l-1}) \oplus H^*(\p^{-\l-3}))
\end{eqnarray*}
which means
$$\H^*(\sF * \sE^{(2)}) \cong \H^*(\sE^{(2)} * \sF) \oplus \sE \otimes_\k H^*(\p^{-\l-2}).$$
This means that in total we have at most $2(-\l-1)$ summands $\sE$ (contradiction). Here we used that $\H^*(\sE^{(2)} * \sF)$ contains no copies $\sE$ which follows by Lemma \ref{lem:supp} since $\supp{\sE} \not\subset \supp(\sE^{(2)} * \sF)$. This proves the deformed claim. 

Finally, we have the map
$$\sE * \sF \oplus \O_\Delta \otimes_\k H^\star(\p^{-\l-1}) \xrightarrow{I \oplus \iota \oplus (I \gamma) \circ \iota \oplus \dots \oplus (I \gamma)^{-\l-1} \circ \iota} \sE * \sF \oplus \sP \cong \sF * \sE(\l+1)$$
where $\iota: \O_\Delta [-\l-1]\{\l+1\} \rightarrow \sP$ is the inclusion into the lowest cohomological degree. By the result above this map is an isomophism at the level of cohomology (so it must be an isomorphism in the derived category).
\end{proof}

\begin{Lemma}\label{lem:supp} If $\l \le 0$ and $r \ge 1$ we have 
$$\supp(\sE^{(r)}(\l-r)) \not\subset \supp(\sE^{(r+k)} * \sF^{(k)}(\l-2r-k))$$ 
and similarly if $\l \ge 0$. 
\end{Lemma}
\begin{proof}
We have that 
$$\supp(\sE^{(r+k)} * \sF^{(k)}) \subset \supp(\sE^{(r+k)}) * \supp(\sF^{(k)})$$
where the right hand side is the set-theoretic convolution of varieties. Also we have 
$$\pi_{Y(\l)}((\supp(\sE^{(r+k)}) * \supp(\sF^{(k)})) \subset \pi_{Y(\l)}(\sE^{(r+k)})$$
where $\pi_{Y(\l)}$ is the projection onto $Y(\l)$. By condition (vii) of having a geometric categorical $\sl_2$ action we have that $\pi_{Y(\l)}(\sE^{(r)}) \not\subset \pi_{Y(\l)}(\sE^{(r+k)})$ and the result follows. 
\end{proof}

\subsection{Idempotent completeness}

Let ${\mathcal C}$ be a graded additive category over $\k$ which is idempotent complete (meaning that every idempotent splits). Notice that the (derived) category of coherent sheaves on any variety is idempotent complete. 

Suppose that (each graded piece of) the space of homs between two objects is finite dimensional. Then every object in ${\mathcal C}$ has a unique, up to isomorphism, direct sum decomposition into indecomposables (see section 2.2 of \cite{Ri}). In particular, this means that if $A,B,C \in {\mathcal C}$ then we have the following cancellation laws: 
\begin{eqnarray}\label{eq:A}
A \oplus B \cong A \oplus C \Rightarrow B \cong C
\end{eqnarray}
\begin{eqnarray}\label{eq:B}
A \otimes_\k \k^n \cong B \otimes_\k \k^n \Rightarrow A \cong B.
\end{eqnarray}

\begin{Corollary} \label{cor:Ecomps}
We have
$$\sE^{(r_2)}(\l+r_1) * \sE^{(r_1)}(\l-r_2) \cong \sE^{(r_1+r_2)}(\l) \otimes_\k H^\star(\bG(r_1,r_1+r_2)).$$
This relation also holds if we replace $\sE$ by $\sF$.
\end{Corollary}
\begin{proof}
Applying the isomorphism from Proposition \ref{prop:theta_k} repeatedly we find that 
\begin{equation*}
\sE(\l+r_1+r_2-1) * \dots * \sE(\l-r_1-r_2+1) \cong \sE^{(r_1+r_2)}(\l) \otimes_\k H^\star(\Fl_{r_1+r_2})
\end{equation*}
where $\Fl_{r_1+r_2}$ denotes the complete flag of $\C^{r_1+r_2}$. Similarly, one finds that 
$$\sE(\l+r_1-r_2-1) * \dots * \sE(\l-r_1-r_2+1) \cong \sE^{(r_1)}(\l-r_2) \otimes_{\k} H^\star(\Fl_{r_1})$$
and
$$\sE(\l+r_1+r_2-1) * \dots * \sE(\l+r_1-r_2+1) \cong \sE^{(r_2)}(\l+r_1) \otimes_{\k} H^\star(\Fl_{r_2}).$$
Thus 
$$\sE^{(r_1+r_2)}(\l) \otimes_\k H^\star(\Fl_{r_1+r_2}) \cong \sE^{(r_2)}(\l+r_1) * \sE^{(r_1)}(\l-r_2) \otimes_\k H^\star(\Fl_{r_1}) \otimes_\k  H^\star(\Fl_{r_2}).$$
But 
$$H^\star(\Fl_{r_1+r_2}) \cong H^\star(\Fl_{r_1}) \otimes_{\k} H^\star(\bG(r_1,r_2+r_2)) \otimes_{\k} H^\star(\Fl_{r_2})$$
and $D(Y(\l-r_1-r_2) \times Y(\l+r_1+r_2))$ is idempotent complete so we can cancel to get our result. The analogous claim for $\sF$'s follows by taking adjoints. 
\end{proof}

\begin{Corollary} \label{cor2:EF=FE}
If $-\l - a + b \ge 0$ we have
$$\sF^{(b)}(\l+2a-b) * \sE^{(a)}(\l+a) \cong \bigoplus_{j \ge 0} \sE^{(a-j)} * \sF^{(b-j)} \otimes_{\k} H^\star(\bG(j, -\l-a+b))$$
where on the right-hand side $\End(\sE^{(a-j)} * \sF^{(b-j)}) \cong \k \cdot I$. Also $\End(\sE^{(a)}(\l-2b+a) * \sF^{(b)}(\l-b)) = \k \cdot I$.

Similarly, if $\l - a + b \ge 0$ then we have 
$$\sE^{(b)}(\l-2a+b) * \sF^{(a)}(\l-a) \cong \bigoplus_{j \ge 0} \sF^{(a-j)} * \sE^{(b-j)} \otimes_{\k} H^\star(\bG(j, \l-a+b))$$
where on the right-hand side $\End(\sF^{(a-j)} * \sE^{(b-j)}) \cong \k \cdot I$. Also $\End(\sF^{(a)}(\l+2b-a) * \sE^{(b)}(\l+b)) \cong \k \cdot I.$
\end{Corollary}
\begin{proof}
We only state this result for completeness since we will not really use it. It follows formally from Proposition \ref{prop2:EF=FE} (see Lemma 4.2 of \cite{ckl2} for a sketch of the proof). 
\end{proof}

We end with the following Lemma which proves the last condition for having a strong categorical $\sl_2$ action.
\begin{Lemma}\label{lem:homs2}
$\Ext^i(\sE^{(r)}, \sE^{(r)}\{j\}) = 0$ for $i < 0$ (and any $j \in \Z$) while $\End(\sE^{(r)}) \cong \k \cdot I$. The same holds with $\sF$'s instead of $\sE$'s. 
\end{Lemma}
\begin{proof}
Notice that the vanishing for $i < 0$ is immediate since every $\sE^{(r)}$ is a sheaf. However, we can avoid using this fact and proceed by induction, reducing everything to the case when $r=0$ (i.e. to the fact that $\End(\O_\Delta)) \cong \k \cdot I$). 

Suppose that $\l \le 0$ ($\l \ge 0$ is done similarly). We will ignore the $\{ \cdot \}$ shifts for notational simplicity. Then 
\begin{align*}
\Ext^i(&\sE^{(r)}(\l+r), \sE^{(r)}(\l+r))  
\cong \Ext^i(\sF^{(r)}(\l+r) * \sE^{(r)}(\l+r) [-r(\l+r)], \O_\Delta) \\
&\cong \Ext^i(\bigoplus_{j \ge 0} \sE^{(r-j)} * \sF^{(r-j)}(\l-r+j) \otimes_{\k} H^\star(\bG(j, -\l)), \O_\Delta [r(\l+r)]) \\
&\cong \bigoplus_{j \ge 0} \Ext^i(\sF^{(r-j)}, \sF^{(r-j)}(\l-r+j) [(r-j)(\l-r+j)] \otimes_{\k} H^\star(\bG(j, -\l)) [r(\l+r)]) \\
&\cong \bigoplus_{j \ge 0} \Ext^i(\sE^{(r-j)}, \sE^{(r-j)} \otimes_{\k} H^\star(\bG(j, -\l)) [(\l+j)(2r-j)]) 
\end{align*}
where we used Corollary \ref{cor2:EF=FE} to obtain the second isomorphism. Using that $0 \le j \le r$ one can show that the right hand side lies in negative degrees unless $j=r$ in which case we get one term in degree zero. So, by induction, if $i < 0$ this vanishes and if $i=0$ we get $\Hom(\O_\Delta, \O_\Delta)$ which is one-dimensional. 

Since $\sF$'s are adjoint to $\sE$'s (up to a shift) the same holds if we replace all the $\sE$'s by $\sF$'s. 
\end{proof}

\section{Proof of nil affine Hecke relations} \label{se:proofofnilHecke}

In this section we prove the following result which, together with the results in section \ref{se:formality}, prove the main Theorem \ref{th:main}. 

\begin{Theorem}\label{th:nilHeckerels}
Given a geometric categorical $\sl_2$ action there exist morphisms
$$X(\l): \sE(\l)\la-1\ra \rightarrow \sE(\l)\la1\ra$$ 
and 
$$T(\l): \sE(\l+1) * \sE(\l-1) \la1\ra \rightarrow \sE(\l+1) * \sE(\l-1) \la-1\ra$$ 
satisfying the nil affine Hecke relations 
\begin{enumerate}
\item $T(\l)^2 = 0$
\item $(I * T(\l-1)) \circ (T(\l+1) * I) \circ (I * T(\l-1)) = (T(\l+1) * I) \circ (I * T(\l-1)) \circ (T(\l+1) * I)$ as morphisms $ \sE(\lambda + 2) * \sE(\lambda) * \sE(\lambda -2)\la3\ra \rightarrow \sE(\lambda + 2) * \sE(\lambda) * \sE(\lambda -2)\la-3\ra $.
\item $(X(\l+1) * I) \circ T(\l) - T(\l) \circ (I * X(\l-1)) = I = - (I * X(\l-1)) \circ T(\l) + T(\l) \circ (X(\l+1) * I)$ as morphisms $ \sE(\lambda +1) * \sE(\lambda -1) \rightarrow \sE(\lambda +1) * \sE(\l-1) $.
\end{enumerate}
The freedom in choosing such $X$s and $T$s is parametrized by 
$$V(-1)^{tr} \times V(-2)^{tr} \times \k^\times \cong V(1)^{tr} \times V(2)^{tr} \times \k^\times$$
where $V(\l)^{tr} \subset \Hom(\O_\Delta \la-1\ra, \O_\Delta \la1\ra)$ denotes the linear subspace of transient maps defined below. 
\end{Theorem}

In this section, we use the notation $ \la k \ra $ for $ [k]\{-k\} $.

\subsection{Definition of $\theta(\l)$}

If we have a geometric categorical $\k^\times$-equivariant $\sl_2$ action then $\tY(\l)$ induces two deformations 
$$i_1: Y(\l) \times Y(\l) \rightarrow \tY(\l) \times Y(\l) \text{ and } i_2: Y(\l) \times Y(\l) \rightarrow Y(\l) \times \tY(\l).$$
The connecting morphism for 
$$\O_\Delta [1]\{-2\} \rightarrow i_j^* i_{j*} \O_\Delta \rightarrow \O_\Delta$$
(for $j=1,2$) give two maps $\alpha_1, \alpha_2: \O_\Delta[-1]\{1\} \rightarrow \O_\Delta[1]\{-1\}$.

Now consider $\sE(\l+1) \in D(Y(\l) \times Y(\l+2))$. By Lemma \ref{cor:j1} below
\begin{equation*}\label{eq:20}
\sE(\l+1) * \Cone(\O_\Delta[-1]\{1\} \xrightarrow{\alpha_1} \O_\Delta[1]\{-1\}) \cong i^* i_* \sE(\l+1) \{1\}
\end{equation*}
where $i$ is the inclusion $Y(\l) \times Y(\l+2) \rightarrow \tY(\l) \times Y(\l+2)$. Notice that we could just as well have used $\alpha_2$ instead of $\alpha_1$.  

An analogous argument shows that 
$$\Cone(\O_\Delta[-1]\{1\} \xrightarrow{\alpha_1} \O_\Delta[1]\{-1\}) * \sE(\l-1) \cong i'^* i'_* \sE(\l-1) \{1\}$$
where $i'$ is the inclusion $Y(\l-2) \times Y(\l) \rightarrow Y(\l-2) \times \tY(\l)$. Combining these two results we can define $\theta(\l)$:  

\begin{Definition} For each $\l$ let $\theta(\l) := \alpha_1$. Then we find that
$$\sE(\l+1) * \Cone(\O_\Delta[-1]\{1\} \xrightarrow{\theta(\l)} \O_\Delta[1]\{-1\}) \cong i^* i_* \sE(\l+1) \{1\}$$
where $i: Y(\l) \times Y(\l+2) \rightarrow \tY(\l) \times Y(\l+2)$ and
$$\Cone(\O_\Delta[-1]\{1\} \xrightarrow{\theta(\l)} \O_\Delta[1]\{-1\}) * \sE(\l-1) \cong i'^* i'_* \sE(\l-1) \{1\}$$
where $i': Y(\l-2) \times Y(\l) \rightarrow Y(\l-2) \times \tY(\l)$. 
\end{Definition}

\begin{Remark}
Proposition \ref{prop:theta_k} implies that the map 
$$\sE * (\O_\Delta[-1]\{1\} \xrightarrow{\theta(\l)} \O_\Delta[1]\{-1\}) * \sE^{(r)}$$
induces an isomorphism on all but two summands $\sE^{(r+1)}$ (one on each side). Similarly, Proposition \ref{prop2:EF=FE} implies that the map
$$\sF(\l-1) * (\O_\Delta[-1]\{1\} \xrightarrow{\theta(\l)} \O_\Delta[1]\{-1\}) * \sE(\l-1)$$
induces an isomorphism on all but two summands $\O_\Delta$ (one on each side). We will use these facts in the future.
\end{Remark}

\begin{Lemma}\label{cor:j1}
Consider varieties $Y$ and $Y'$ and the deformation $\tY \rightarrow \bA^1_\k$ of $Y$ (all equipped with compatible $\k^\times$ actions). Let $\sG \in D(Y \times Y')$ be some kernel. Then 
$$\sG * \Cone(\O_\Delta[-1]\{1\} \xrightarrow{\alpha_j} \O_\Delta[1]\{-1\}) \cong  i^* i_* \sG \{1\}.$$
Here $i$ is the inclusion $Y \times Y' \rightarrow \tY \times Y'$ and $\alpha_1, \alpha_2$ are the connecting maps for $\O_\Delta$ associated to the two deformations $\tY \times Y$ and $Y \times \tY$ of $Y \times Y$. 
\end{Lemma}
\begin{proof}
First let us consider $\alpha_2$. We will be working on $Y \times Y \times Y'$ so let $i_{12}: Y \times Y \rightarrow Y \times \tY$ and $i = i_{23}: Y \times Y' \rightarrow \tY \times Y'$ denote the natural inclusions. 

Let $ \beta_{2} = \sG * \alpha_2 : \sG[-1]\{1\} \rightarrow \sG[1]\{-1\}$. Then by Lemma \ref{lem:j1}
\begin{align*}
\sG * \Cone(\O_\Delta[-1]\{1\} \xrightarrow{\alpha_2} \O_\Delta[1]\{-1\}) 
&\cong \sG * i_{12}^* i_{12*} \O_\Delta \{1\} \\
&\cong i_{23}^* i_{23*} \sG * \O_\Delta \{1\} \\
&\cong i_{23}^* i_{23*} \sG \{1\}.
\end{align*}
Hence we get 
$$\Cone(\sG [-1]\{1\} \xrightarrow{\beta_2} \sG [1]\{-1\} ) \cong i^* i_* \sG \{1\}.$$

The case of $\alpha_1$ follows similarly. Let $i'_{12}: Y \times Y \rightarrow \tY \rightarrow Y$ be the natural inclusion. Then 
\begin{align*}
\sG * \Cone(\O_\Delta[-1]\{1\} \xrightarrow{\alpha_1} \O_\Delta[1]\{-1\})
&\cong \sG * i_{12}^{'*} i'_{12*} \O_{\Delta} \{1\} \\
&\cong \sG * i_{12}^{'*} i'_{12*} \O_{\Delta} * \O_\Delta \{1\} \\
&\cong \sG * \O_\Delta * i_{12}^{*} i_{12*} \O_{\Delta} \{1\} \\
&\cong \sG * i_{12}^* i_{12*} \O_\Delta \{1\} \\
&\cong i_{23}^* i_{23*} \sG \{1\}
\end{align*}
where the last isomorphism follows as above. 
\end{proof}

\subsection{Construction and proof of Relation (iii)}\label{sec:reliii}

We first construct $X$s and $T$s satisfying nil affine Hecke relation (iii).  

\subsubsection{$\theta(\l)$ and Transient Maps}

The first step is to better understand maps $\sE(\l)\la-1\ra \rightarrow \sE(\l)\la1\ra$. We will write $\Delta(\l)$ for the diagonal inside $Y(\l) \times Y(\l)$ when we want to emphasize where the diagonal lives. For convenience we will assume from now on that $Y(\l) = \emptyset$ either for all $\l$ odd or all $\l$ even. 

Without loss of generality we can assume $\l \le 0$. Then 
$$\Hom(\sE(\l), \sE(\l)\la2\ra) \cong \Hom(\O_\Delta, \sE(\l)_R * \sE(\l)\la 2\ra ) \cong \Hom(\O_\Delta, \sF(\l) * \sE(\l) \la\l+2\ra).$$
Now $\sF(\l) * \sE(\l) \cong \sE(\l-2) * \sF(\l-2) \oplus \O_\Delta \otimes_\k H^\star(\p^{-\l})$ and we have that 
\begin{eqnarray*}
\Hom(\O_\Delta, \sE(\l-2) * \sF(\l-2)\la\l+2\ra) &\cong&
\Hom(\O_\Delta, \sF(\l-2)_R * \sF(\l-2) \la2\l\ra) \\
&\cong& \Hom(\sF(\l-2), \sF(\l-2)\la 2\l \ra).
\end{eqnarray*}
is zero if $\l \le -1$. So if $\l \le -1$ we get 
\begin{eqnarray*}
\Hom(\sE(\l), \sE(\l)\la 2 \ra) 
&\cong& \Hom(\O_\Delta, \O_\Delta \otimes_\k H^\star(\p^{-\l})\la\l+1\ra) \\
&\cong& \Hom(\O_{\Delta(\l-1)}, \O_{\Delta(\l-1)}) \oplus \Hom(\O_{\Delta(\l-1)}, \O_{\Delta(\l-1)}\la 2 \ra ).
\end{eqnarray*}
Meanwhile, if $\l=0$ we get 
$$\Hom(\sE(\l), \sE(\l) \la 2 \ra) \cong \Hom(\sF(\l-2), \sF(\l-2)) \oplus \Hom(\O_{\Delta(\l-1)}, \O_{\Delta(\l-1)}\la 2 \ra ).$$

In both cases we have the maps indexed by $\Hom(\O_\Delta, \O_\Delta\la 2 \ra)$. If we examine the adjunction calculation above it is easy to see that these maps correspond to those of the form 
$$\sE(\l) * (\O_{\Delta(\l-1)} \la-1\ra \rightarrow \O_{\Delta(\l-1)}\la1\ra).$$
On the other hand, there is also the map 
$$(\O_{\Delta(\l+1)} \la-1\ra \xrightarrow{\theta(\l+1)} \O_{\Delta(\l+1)}\la1\ra) * \sE(\l).$$
This map cannot be of the form above because if it were then 
\begin{eqnarray*}
& & \sE(\l+2) * (\O_{\Delta(\l+1)} \la-1\ra \xrightarrow{\theta(\l+1)} \O_{\Delta(\l+1)}\la1\ra) * \sE(\l) \\
&\cong& \sE(\l+2) * \sE(\l) * (\O_{\Delta(\l-1)}\la-1\ra \rightarrow \O_{\Delta(\l-1)}\la1\ra) \\
&\cong& \sE^{(2)}(\l+1) * (\O_{\Delta(\l-1)}\la-1\ra \rightarrow \O_{\Delta(\l-1)}\la1\ra) \otimes_{\k} H^\star(\p^1)
\end{eqnarray*}
induces the zero map on the cohomology $\sE^{(2)}(\l+1)\la-1\ra \oplus \sE^{(2)}(\l+1)\la1\ra$ of $\sE(\l+2) * \sE(\l)$ (note that the map itself is not zero). This contradicts Proposition \ref{prop:theta_k} in the case $r=1$. 

We conclude that any map $\sE(\l)\la-1\ra \rightarrow \sE(\l)\la1\ra$ is of the form $a\theta(\l+1) + \phi$ where $\phi: \O_{\Delta(\l-1)}\la-1\ra \rightarrow \O_{\Delta(\l-1)}\la1\ra$ is arbitrary. Notationally we will write this map as $(a \theta(\l+1), \phi)$ to remind ourselves that $\theta(\l+1)$ acts on the left of $\sE(\l)$ while $\phi$ acts on the right. 

There is one last piece of nice structure here worth exploiting. Take any 
$$\phi \in \Hom(\O_{\Delta(\l-1)}\la-1\ra, \O_{\Delta(\l-1)}\la1\ra).$$
Then by the argument above 
$$(\O_{\Delta(\l-1)} \la-1\ra \xrightarrow{\phi} \O_{\Delta(\l-1)}\la1\ra) * \sE(\l-2)$$ 
must be equal to a map $(a' \theta(\l-1), \phi')$ for some $a' \in \k$ and $\phi': \O_{\Delta(\l-3)}\la-1\ra \rightarrow \O_{\Delta(\l-3)}\la1\ra$. This means that there exists a distinguished linear subspace 
$$V(\l-1)^{tr} \subset \Hom(\O_{\Delta(\l-1)}\la-1\ra, \O_{\Delta(\l-1)}\la1\ra)$$ 
consisting of those $\phi$ which induce $a' = 0$. We will call such an element $\phi \in V(\l-1)^{tr}$ a {\emph transient map}. 

We can define transient maps $V(\l)^{tr} \subset \Hom(\O_{\Delta(\l)}\la-1\ra, \O_{\Delta(\l)}\la1\ra)$ for every $Y(\l)$ by using $\sE(\l-1)$ if $\l \le 0$ and $\sF(\l+1)$ if $\l \ge 0$. There is a small conflict when $\l=0$ since there are two ways of defining transient maps in that case. Fortunately the two definitions agree (see Proposition \ref{prop:transients} below). 

Notice that $V(\l)^{tr} \subset \Hom(\O_{\Delta(\l)}\la-1\ra, \O_{\Delta(\l)}\la1\ra)$ is a codimension one linear subspace (the only exception is at the extremes where $\l = \pm N$ in which case every map is transient). If $\phi$ is transient then we can ``slide'' it from the $Y(\l)$ slot to the $Y(\l-2)$ slot if $\l \le 0$ ($Y(\l+2)$ slot if $\l \ge 0$) to obtain some new $\phi'$. As Proposition \ref{prop:transients} below shows $\phi'$ will again be transient and so we can repeat the process. This is why we call them transient maps. To summarize: 

\begin{Proposition}\label{prop:E2maps}
Every map $\sE(\l)\la-1\ra \rightarrow \sE(\l)\la1\ra$ is of the form $(a \theta(\l+1), b \theta(\l-1) + \phi)$ if $\l \le 0$ and $(a \theta(\l+1) + \phi, b \theta(\l-1))$ if $\l \ge 0$ where $a,b \in \k$ and $\phi$ is transient. Taking adjoints we obtain the analogous claim for maps $\sF(\l)\la-1\ra \rightarrow \sF(\l)\la1\ra$. 
\end{Proposition}

\begin{Proposition}\label{prop:transients}
Transient maps are well defined and come equipped with natural maps $V(\l)^{tr} \rightarrow V(\l-2)^{tr}$ if $\l \le 0$ and $V(\l)^{tr} \rightarrow V(\l+2)^{tr}$ if $\l \ge 0$. We also have isomorphisms $V(-1)^{tr} \cong V(1)^{tr}$ and $V(-2)^{tr} \cong V(2)^{tr}$. For $\l \ne \pm N$, the quotient $\Hom(\O_{\Delta(\l)}\la-1\ra, \O_{\Delta(\l)}\la1\ra)/V(\l)^{tr}$, which is one-dimensional, is spanned by $[\theta(\l)]$. If $\l = \pm N$ this quotient is zero. 
\end{Proposition}
\begin{proof}
To obtain the natural maps note that if $\phi \in V(\l)^{tr}$ (where $\l \le 0$) then by definition 
$$(\phi,0): \sE(\l-1)\la-1\ra \rightarrow \sE(\l-1)\la1\ra$$
is isomorphic to $(0, \phi')$ for a unique $\phi' \in V(\l-2)$. What we need to check is that $\phi'$ is transient. 

To this this we consider 
$$(\phi,0,0) = (0, \phi', 0): \sE(\l-1) * \sE(\l-3)\la-1\ra \rightarrow \sE(\l-1) * \sE(\l-3)\la1\ra.$$
We can do this unless $\l-2=-N$ in which case $\phi'$ is automatically transient since all such maps are transient. Then $(0,\phi',0) = (0,a \theta(\l-2), \phi'')$ for some 
$$a \in \k \text{ and } \phi'' \in \Hom(\O_{\Delta(\l-4)}\la-1\ra, \O_{\Delta(\l-4)}\la1\ra).$$

We need to show that $a=0$. To do this we consider the cone 
$$\Cone(\sE(\l-1) * \sE(\l-3)\la-1\ra \xrightarrow{(0,a \theta(\l-2), \phi'')} \sE(\l-1) * \sE(\l-3)\la1\ra)$$
If $a \ne 0$ then the induced map on cohomology is an isomorphism in homological degree zero and so the cone has non-zero cohomology only in degrees $-2$ and $1$. But this cone is the same $\Cone(\phi,0,0)$ which does not induce an isomorphism in degree zero and hence has non-zero cohomology in homological degrees $-2,-1,0,1$. Thus $a=0$ and $\phi'$ is transient. 

The case of $\l \ge 0$ follows similarly. 

That $V(-1)^{tr} \cong V(1)^{tr}$ follows from the fact that we have natural maps $V(-1)^{tr} \rightarrow V(1)^{tr} \rightarrow V(-1)^{tr}$ whose composition is the identity. To see that the composition is the identity observe that a map 
$$\sE(0) \la-1\ra \rightarrow \sE(0) \la1\ra$$
is of the form $(a \theta(1), b \theta(-1) + \phi)$ for a {\em unique} transient $\phi$. 

The isomorphism $V(-2)^{tr} \cong V(2)^{tr}$ follows similarly by looking at maps $\sE^{(2)}(0)\la-1\ra \rightarrow \sE^{(2)}(0)\la1\ra$ and using Proposition \ref{prop:E22maps} which states that any such map is of the form $(a \theta(2), b \theta(-2) + \phi)$ for a unique transient $\phi$. 
\end{proof}

\subsubsection{Defining the $T$s and $X$s modulo transient maps}

As a first step we will define the $X(\l)$s up to transients. Working modulo transients is more convenient since (for $\l \ne \pm N$)
$$\Hom(\sE(\l), \sE(\l)\la 2 \ra) \mbox{ modulo transients } \cong \{(a \theta(\l+1), b \theta(\l-1)): a,b \in \k \}$$
is two-dimensional spanned by $(0,\theta(\l-1))$ and $(\theta(\l+1),0)$. Thus to determine $X(\l)$ modulo transients we only need to choose $a(\l),b(\l) \in \k^2$ and define $X(\l) := (a(\l) \theta(\l+1), b(\l) \theta(\l-1))$. 

We begin by fixing an isomorphism
\begin{equation}\label{eq:iso2}
\sE(\l+1) * \sE(\l-1) \xrightarrow{\sim} \sE^{(2)}(\l)\la-1\ra \oplus \sE^{(2)}(\l)\la1\ra.
\end{equation}
This isomorphism is not unique since we can compose it with elements of 
$$\Aut(\sE^{(2)}(\l)\la-1\ra \oplus \sE^{(2)}(\l)\la1\ra) \cong 
\left\{ \left( \begin{matrix} a \cdot I & 0 \\ \alpha & b \cdot I \end{matrix} \right): a,b \in \k^\times, \alpha \in \Hom(\sE^{(2)}(\l)\la-1\ra, \sE^{(2)}(\l)\la1\ra) \right\}
$$
(here we use that $\Ext^i(\sE^{(k)}(\l), \sE^{(k)}(\l)\{j\}) = 0$ for $i < 0$ and any $j \in \Z$ while $\End(\sE^{(k)}(\l)) \cong \k \cdot I$). 

Using this isomorphism we can write 
$$I * X(\l-1) = \left( \begin{matrix} A & B \\ C & D \end{matrix} \right): 
\left( \begin{matrix} \sE^{(2)}(\l)\la-2\ra \\ \sE^{(2)}(\l) \end{matrix} \right) \rightarrow 
\left( \begin{matrix} \sE^{(2)}(\l) \\ \sE^{(2)}(\l)\la2\ra \end{matrix} \right)$$
where $A,D \in \Hom(\sE^{(2)}(\l), \sE^{(2)}(\l)\la2\ra)$, $B \in \End(\sE^{(2)}(\l))$ and $C \in \Hom(\sE^{(2)}(\l), \sE^{(2)}(\l)\la4\ra)$. Similarly, we have
$$X(\l+1) * I = \left( \begin{matrix} A' & B' \\ C' & D' \end{matrix} \right): 
\left( \begin{matrix} \sE^{(2)}(\l)\la-2\ra \\ \sE^{(2)}(\l) \end{matrix} \right) \rightarrow 
\left( \begin{matrix} \sE^{(2)}(\l) \\ \sE^{(2)}(\l)\la2\ra \end{matrix} \right).$$
Note that although these two matrices are defined only up to conjugation their traces $A+D$ and $A'+D'$ as well as $B$ and $B'$ are invariant under conjugation. 

\begin{Lemma}\label{lem:Bnotzero} If $a(\l-1) \ne 0$ then $B$ is a non-zero multiple of $I$. Similarly, if $b(\l+1) \ne 0$ then $B'$ is a non-zero multiple of $I$. 
\end{Lemma}
\begin{proof}
Since $I * X(\l-1) = (0, a(\l-1) \theta(\l), b(\l-1) \theta(\l-2))$ we know that (at the level of cohomology)
\begin{eqnarray*}
\H(\Cone(I * X(\l-1))) 
& \cong & \H(\Cone(0, a(\l-1) \theta(\l), 0) \\
& \cong & \sE^{(2)}(\l) \otimes_\k (\k[-1]\{1\} \oplus \k[2]\{-3\})  
\end{eqnarray*}
since $a(\l-1) \ne 0$. But the long exact sequence in cohomology induced by $I * X(\l-1)$ looks like 
$$\dots \rightarrow 0 \rightarrow \sE^{(2)}(\l) \xrightarrow{B} \sE^{(2)}(\l) \rightarrow 0 \rightarrow \dots$$
so that $B$ has to be an isomorphism. Since $\End(\sE^{(2)}(\l)) = \k \cdot I$ the result follows. 

The result for $B'$ follows similarly. 
\end{proof}

At this point we define 
\begin{equation}\label{eq:T}
T(\l) := \left( \begin{matrix} 0 & 0 \\ -B^{-1} & 0 \end{matrix} \right): 
\left( \begin{matrix} \sE^{(2)}(\l) \\ \sE^{(2)}(\l)\la-2\ra \end{matrix} \right) \rightarrow
\left( \begin{matrix} \sE^{(2)}(\l)\la2\ra \\ \sE^{(2)}(\l) \end{matrix} \right).
\end{equation}
Notice that this map is invariant under conjugation and hence does not depend on our choice of isomorphism (\ref{eq:iso2}). In this notation it is now easy to check that nil affine Hecke relation (iii) is equivalent to the conditions
$$B+B'=0=C+C' \text{ and } A=D' \text{ and } A'=D.$$

\begin{Remark}\label{rem:nilhecke} A second way to characterize nil affine Hecke relation (iii) is by the conditions that $I * X(\l-1) + X(\l+1) * I$ is a multiple of the identity and 
$$\trace(I * X(\l-1)) = \trace(X(\l+1) * I): \sE^{(2)}(\l)\la-1\ra \rightarrow \sE^{(2)}(\l)\la1\ra.$$
This second condition can also be replaced by asking that 
$$X(\l+1) * X(\l-1): \sE(\l+1) * \sE(\l-1) \la-2\ra \rightarrow \sE(\l+1) * \sE(\l-1) \la2\ra$$
is diagonal. 
\end{Remark} 

We will now recursively define the $X$s. As a first step we let $b(\l+1) = -a(\l-1)$. Then we begin with the smallest weight by first defining 
$$X(-N+1) := (\theta(-N+2),0): \sE(-N+1)\la-1\ra \rightarrow \sE(-N+1)\la1\ra.$$
Notice that on $Y(-N)$ all maps $\O_\Delta\la-1\ra \rightarrow \O_\Delta\la1\ra$ are transient so the only choice we have is which (non-zero) multiple of $\theta(-N+2)$ we should take. Clearly the space of such choices is parametrized by $\k^\times$. 

Now suppose by induction that we have defined $X(-N+1), \dots, X(\l-1), X(\l+1) = (a(\l+1) \theta(\l+2), -a(\l-1) \theta(\l))$ such that nil affine Hecke relation (iii) holds for every pair up to $\sE(\l+1) * \sE(\l-1)$ and such that all the $a$'s are non-zero. This means 
$$X(\l+3) := (a(\l+3) \theta(\l+4), -a(\l+1) \theta(\l+2)): \sE(\l+3)\la-1\ra \rightarrow \sE(\l+3)\la1\ra$$
where $a(\l+3)$ remains to be determined.  

Let's fix again some isomorphism 
\begin{equation}\label{eq:iso3}
\sE(\l+3) * \sE(\l+1) \xrightarrow{\sim} \sE^{(2)}(\l+2)\la-1\ra \oplus \sE^{(2)}(\l+2)\la1\ra
\end{equation}
under which we have the identifications
$$I * X(\l+1) = \left( \begin{matrix} \hA & \hB \\ \hC & \hD \end{matrix} \right) \text{ and } 
X(\l+3) * I = \left( \begin{matrix} \hA' & \hB' \\ \hC' & \hD' \end{matrix} \right).$$

\begin{Lemma}\label{prop:lemX} We have $\hB+\hB'=0=\hC+\hC'$ while 
$$\hA+\hA'=\hD+\hD'=(a(\l+3) \theta(\l+4), b(\l+1) \theta(\l)): \sE^{(2)}(\l+2) \la-1\ra \rightarrow \sE^{(2)}(\l+2)\la1\ra.$$
\end{Lemma}
\begin{proof}
We have 
$$I * X(\l+1) = (0,  a(\l+1) \theta(\l+2), -a(\l-1) \theta(\l))$$ 
and 
$$X(\l+3) * I = (a(\l+3) \theta(\l+4), -a(\l+1) \theta(\l+2),0)).$$
Hence 
$$I * X(\l+1) + X(\l+3) * I = (a(\l+3) \theta(\l+4), 0, -a(\l-1) \theta(\l))$$
and the result follows. 
\end{proof}

It remains to show that we can choose $a(\l+3) \ne 0$ such that $\hA=\hD'$ and $\hA'=\hD$. To do this we first need to understand the possible maps $\sE^{(2)}(\l)\la-1\ra \rightarrow \sE^{(2)}(\l)\la1\ra$. 

\begin{Proposition}\label{prop:E22maps}
Every map $\sE^{(2)}(\l)\la-1\ra \rightarrow \sE^{(2)}(\l)\la1\ra$ is of the form $(a \theta(\l+2), b \theta(\l-2) + \phi)$ if $\l \le 0$ and $(a \theta(\l+2) + \phi, b \theta(\l-2))$ if $\l \ge 0$ where $a,b \in \k$ and $\phi$ is transient. 
\end{Proposition}
\begin{proof}
This result is analogous to Proposition \ref{prop:E2maps} and the proof is very similar. 

Suppose $\l \le 0$. Then 
\begin{align*}
\Hom(\sE^{(2)}(\l), \sE^{(2)}(\l)\la2\ra) 
&\cong \Hom(\O_\Delta, \sE^{(2)}(\l)_R * \sE^{(2)}(\l)\la2\ra) \\
&\cong \Hom(\O_\Delta, \sF^{(2)}(\l) * \sE^{(2)}(\l) \la2\l+2\ra).
\end{align*}
Now by Corollary \ref{cor2:EF=FE}
\begin{align*}
\sF^{(2)}(\l) * \sE^{(2)}(\l) \cong \sE^{(2)}(\l-4) * \sF^{(2)}(\l-4) 
&\oplus \sE(\l-3) * \sF(\l-3) \otimes_\k H^\star(\bG(1,-\l+2)) \\
&\oplus \O_\Delta \otimes_\k H^\star(\bG(2, -\l+2)). 
\end{align*}
Now 
\begin{align*}
\Hom(\O_\Delta, \sE^{(2)}(\l-4) * \sF^{(2)}(\l-4)\la2\l+2\ra) 
&\cong \Hom(\O_\Delta, \sF^{(2)}(\l-4)_R * \sF^{(2)}(\l-4)\la4\l-6)\ra) \\
&\cong \Hom(\sF^{(2)}(\l-4), \sF^{(2)}(\l-4) \la4\l-6\ra) 
\end{align*}
is zero for $\l \le 0$. Also 
\begin{eqnarray*}
\Hom(\O_\Delta, \sE(\l-3) * \sF(\l-3)) 
&\cong& \Hom(\O_\Delta, \sF(\l-3)_R * \sF(\l-3)\la\l-3\ra) \\
&\cong& \Hom(\sF(\l-3), \sF(\l-3) \la\l-3\ra)
\end{eqnarray*}
so that 
\begin{equation*}
 \Hom(\O_\Delta, \sE(\l-3) * \sF(\l-3) \otimes_\k H^\star(\bG(1,-\l+2) \la2\l+2\ra) 
\cong \Hom(\sF(\l-3), \sF(\l-3) \otimes_\k H^\star(\p^{-\l+1}) \la3\l-1\ra)
\end{equation*}
is zero if $\l < 0$ and isomorphic to $\Hom(\sF(-3), \sF(-3) \otimes_\k H^\star(\p^1) \la1\ra) \cong \k$ if $\l=0$. 

Finally, notice that $H^\star(\bG(2,-\l+2))$ is supported in homological degrees $\ge 2\l$ and one-dimensional in lowest homological degrees $2\l$ and $2\l+2$ (remember $\l \le 0$). It follows that if $\l < 0$ 
$$\Hom(\O_\Delta, \O_\Delta \otimes_\k H^\star(\bG(2, -\l+2)\la2\l+2\ra) \cong \Hom(\O_\Delta,\O_\Delta) \oplus \Hom(\O_\Delta, \O_\Delta \la2\ra)$$
while if $\l = 0$ we get $\Hom(\O_\Delta, \O_\Delta \otimes_\k H^\star(\bG(2,2)\la2\ra) \cong \Hom(\O_\Delta, \O_\Delta\la2\ra)$. 

So if $\l \le 0$ the space of maps $\Hom(\sE^{(2)}(\l), \sE^{(2)}(\l)\la2\ra)$ is spanned by maps of the form $(0, b \theta(\l-2) + \phi)$ (corresponding to the factor $\Hom(\O_\Delta, \O_\Delta \la2\ra)$ in the calculation above) and one more map (corresponding to the factor $\Hom(\O_\Delta, \O_\Delta) \cong \k$). Since $(\theta(\l+2),0)$ is linearly independent to the maps above (as a corollary of Proposition \ref{prop:theta_k}) we can use it as the extra map. 

The proof for $\l \ge 0$ is analogous. 
\end{proof}

\begin{Proposition}\label{prop:defX} There exists a unique $a(\l+3) \ne 0$ so that $\hA=\hD'$ and $\hA'=\hD$ in the notation above.
\end{Proposition}
\begin{proof}
Choose $a(\l+3)$ arbitrarily and consider the map 
\begin{eqnarray*}
X(\l+3) * X(\l+1) * X(\l-1) 
&:& \sE(\l+3) * \sE(\l+1) * \sE(\l-1) \la-3\ra \\
& \rightarrow & \sE(\l+3) * \sE(\l+1) * \sE(\l-1) \la3\ra 
\end{eqnarray*}
at the level of cohomology. Recall that 
$$\sE(\l+3) * \sE(\l+1) * \sE(\l-1) \cong \sE^{(3)}(\l+1) \otimes_\k H^\star(\Fl_3).$$

Now 
$$X(\l+1) * X(\l-1) = \left( \begin{matrix} A' & B' \\ C' & D' \end{matrix} \right) \cdot \left( \begin{matrix} A & B \\ C & D \end{matrix} \right) = \left( \begin{matrix} AD-BC & 0 \\ 0 & AD-BC \end{matrix} \right)$$
since by induction we have nil affine Hecke relation (iii) and so $B+B'=0=C+C'$ and $A=D'$ and $A'=D$. This means that 
$$X(\l+3) * X(\l+1) * X(\l-1) = X(\l+3) * (AD-BC) \otimes_\k H^\star(\p^1)$$
where
$$X(\l+3) * (AD-BC): \sE(\l+3) * \sE^{(2)}(\l)\la-3\ra \rightarrow \sE(\l+3) * \sE^{(2)}(\l)\la3\ra.$$
Since $\sE(\l+3) * \sE^{(2)}(\l) \cong \sE^{(3)} \otimes_\k H^\star(\p^2)$ such a map is automatically zero at the level of cohomology (regardless of our choice of $X(\l+3)$). Thus $X(\l+3) * X(\l+1) * X(\l-1)$ induces zero at the level of cohomology. 

On the other hand we can consider 
$$X(\l+3) * X(\l+1) = \left( \begin{matrix} \hA' & \hB' \\ \hC' & \hD' \end{matrix} \right) \cdot \left( \begin{matrix} \hA & \hB \\ \hC & \hD \end{matrix} \right) = \left( \begin{matrix} * & \hB(\hA'-\hD) \\ * & * \end{matrix} \right)
$$
where we use that $\hB+\hB'=0$. Each entry marked $*$ (whose precise value we do not care about) has homological degree four or six. So by degree considerations each entry $*$ induces zero on the cohomology of 
$$\sE(\l+3) * \sE(\l+1) \cong \sE^{(2)}(\l+2) \otimes_\k H^\star(\p^1).$$ 
Hence 
$$\hB(\hA'-\hD) * X(\l-1): \sE^{(2)}(\l+2) * \sE(\l-1)\la-2\ra \rightarrow \sE^{(2)}(\l+2) * \sE(\l-1)\la2\ra$$
must also induce zero at the level of cohomology. Notice that $a(\l+1) \ne 0$ by induction so that $\hB \ne 0$ by Lemma \ref{lem:Bnotzero} and hence $(\hA'-\hD) * X(\l-1)$ must induce zero. 

By proposition \ref{prop:E22maps} we know that (modulo transient maps) $(\hA'-\hD) = (u \theta(\l+4), v \theta(\l))$ for some $u,v \in \k$. Also $X(\l-1) = (a(\l-1)\theta(\l), -a(\l-3) \theta(\l-2))$ where $a(\l-1) \ne 0$. Thus if $v \ne 0$ then by Proposition \ref{prop:theta_k} both
$$(\hA'-\hD) * I \text{ and } I * X(\l-1): \sE^{(3)}(\l+1)\la-1\ra \otimes_\k H^\star(\p^2) \rightarrow \sE^{(3)}(\l+1)\la1\ra \otimes_\k H^\star(\p^2)$$
induce an isomorphism on the two copies of $\sE^{(3)}(\l+1)$ in homological degrees $-1$ and $1$. Consequently, the composition 
$$((\hA'-\hD) * I) \circ (I * X(\l-1)) = (\hA'-\hD) * X(\l-1)$$
would induce an isomorphism on one copy of $\sE^{(3)}(\l+1)$. But we showed above this is not the case so $v=0$.

Now notice that the map 
$$(\theta(\l+4), 0, 0): \sE(\l+3) * \sE(\l+1) \la-1\ra \rightarrow \sE(\l+3) * \sE(\l+1) \la1\ra$$
corresponds to $\left( \begin{matrix} \theta(\l+4) & 0 \\ 0 & \theta(\l+4) \end{matrix} \right)$ under the isomorphism (\ref{eq:iso3}). Thus for $c \in \k$ we get
$$(X(\l+3) + (c \theta(\l+4), 0)) * X(\l+1) = \left( \begin{matrix} * & \hB (u \theta(\l+4), 0) \\ * & * \end{matrix} \right) + \left( \begin{matrix} * & \hB c \theta(\l+4) \\ * & * \end{matrix} \right).$$

So if we take $c=-u$ and replace $X(\l+3)$ by $X(\l+3) + (c \theta(\l+4),0)$ then we get that $\hA'=\hD$. Since $\hA+\hA'=\hA+\hD'$ we also get $\hA=\hD'$ and hence there exists a unique $a(\l+3)$ as required. 

The only thing left is to prove that $a(\l+3) \ne 0$. To do this consider 
$$\sE(\l+5) * \sE(\l+3) * \sE(\l+1).$$
Note that if $\sE(\l+5) = 0$ (i.e. $\l+5 \ge N$ so we are past the highest weight) then $Y(\l+4) = Y(N)$ and hence $\theta(\l+4)=0$ so there is nothing to prove. 

By construction we know that $X(\l+3) * X(\l+1)$ is diagonal. From the argument above we know this means that 
\begin{equation}\label{eq:iso4}
X(\l+5) * X(\l+3) = \left( \begin{matrix} * & (u' \theta(\l+6), 0) \\ * & * \end{matrix} \right)
\end{equation} 
for any $X(\l+5)$ we like. But if $a(\l+3) = 0$ then 
$$I * X(\l+3) = (0, 0,  -a(\l+1) \theta(\l+2)) = \left( \begin{matrix} (0,-a(\l+1) \theta(\l+2)) & 0 \\ 0 & (0,-a(\l+1) \theta(\l+2)) \end{matrix} \right)$$
while we can take
$$X(\l+5) * I  = (0, \theta(\l+4), 0) = \left( \begin{matrix} * & \beta \\ * & * \end{matrix} \right)$$
where $\beta \in \k^\times$. Then 
$$X(\l+5) *  X(\l+3) = \left( \begin{matrix} * & \beta (0,-a(\l+1) \theta(\l+2)) \\ * & * \end{matrix} \right)$$
contradicting equation (\ref{eq:iso4}). Thus $a(\l+3) \ne 0$ and we are done. 
\end{proof}

Thus repeatedly using Proposition \ref{prop:defX} we find that:

\begin{Corollary}\label{cor:defX} There exist non-zero $a$s such that the $X$s defined by 
$$X(\l+1) := (a(\l+1) \theta(\l+2), -a(\l-1) \theta(\l))$$ 
together with the $T$s defined by equation (\ref{eq:T}) satisfy nil affine Hecke relation (iii) (modulo transients). 
\end{Corollary}

\subsubsection{Defining the $X$s on the nose}

At this point we can choose our $X$s and $T$s so that they satisfy nil affine Hecke relation (iii) modulo transients. We will now readjust these $X$s by appropriate transients so that relation (iii) holds on the nose. 

If $N$ is odd we start with $X(0)$ which we keep unchanged. Now any map 
$$(\phi(1), \phi(-1), \phi(-3)): \sE(0) * \sE(-2) \la-1\ra \rightarrow \sE(0) * \sE(-2) \la1\ra$$
where the $\phi$ is a transient map is equivalent to a map $(0, 0, \phi)$ since we can slide over transient maps. So under the isomorphism
$$\sE(0) * \sE(-2) \xrightarrow{\sim} \sE^{(2)}(-1)\la-1\ra \oplus \sE^{(2)}(-1)\la1\ra$$
any combination of transient maps is of the form $\left( \begin{matrix} \phi & 0 \\ 0 & \phi \end{matrix} \right)$. 

Since $X(0)$ and $X(-2)$ satisfy nil affine Hecke relation (iii) modulo transients we conclude that 
$$X(0) * I = \left( \begin{matrix} A & B \\ C & D \end{matrix} \right) \text{ and } 
I * X(-2) = \left( \begin{matrix} D + \phi & -B \\ -C & A + \phi \end{matrix} \right)$$
for some transient map $\phi: \O_{\Delta(-3)}\la-1\ra \rightarrow \O_{\Delta(-3)}\la1\ra$. 

So if we replace $X(-2) = (a(-2) \theta(-1), b(-2) \theta(-3))$ by 
$$X(-2) := (a(-2) \theta(-1), b(-2) \theta(-3) - \phi)$$ 
then we get nil affine Hecke relation (iii) on the nose. Notice that the $\phi$ by which we had to change $X(-2)$ is completely determined. Now we can repeat with $X(-4), X(-6), \dots X(-N)$ and then similarly with $X(2), X(4), \dots, X(N)$. The overall freedom we had in redefining the $X$s comes from being able to choose $X(0)$ arbitrarily. This choice is parametrized by $V(1)^{tr} \cong V(-1)^{tr}$. 

If $N$ is even we do the same thing except starting with $X(-1)$. We then recursively redefine $X(-3), \dots, X(-N)$ as above followed by $X(1), \dots, X(N)$. This time the freedom we have in redefining the $X$s comes from being able to choose $X(-1)$ arbitrarily. This choice is parametrized by $V(-2)^{tr}$. Notice that by symmetry we could have started with $X(1)$ and then the freedom would have been parametrized by $V(2)^{tr}$ but by Proposition \ref{prop:transients} $V(-2)^{tr} \cong V(2)^{tr}$.

This completes the proof of nil affine Hecke relation (iii) in Theorem \ref{th:nilHeckerels} as well as the proof regarding the freedom we have in choosing the $X$s and $T$s. 

\subsection{Proof of nil affine Hecke Relations (i) and (ii)}

Nil Hecke relations (i) and (ii) now follow fairly easily from relation (iii). 

Relation (i) is immediate. Notice that 
$$T(\l)^2: \sE^{(2)}(\l) \otimes_\k (\k\la1\ra \oplus \k\la3\ra) \rightarrow \sE^{(2)}(\l) \otimes_\k (\k\la-3\ra \oplus \k\la-1\ra)$$ 
and there are no negative homological degree endomorphisms of $\sE^{(2)}$ (since it is a sheaf) so $T(\l)^2=0$. 

To prove relation (ii) note that 
$$\sE(\l+2) * \sE(\l) * \sE(\l-2) \cong \sE^{(3)}(\l) \otimes \left( \k\la-3\ra \oplus \k\la-1\ra^{\oplus 2} \oplus \k\la1\ra^{\oplus 2} \oplus \k\la3\ra \right).$$
Since $\End(\sE^{(3)}(\l)) \cong \k \cdot I$ this means
$$\Hom(\sE(\l+2) * \sE(\l) * \sE(\l-2)\la3\ra, \sE(\l+2) * \sE(\l) * \sE(\l-2) \la-3\ra) \cong \k.$$
Thus $(I * T(\l-1)) \circ (T(\l+1) * I) \circ (I * T(\l-1))$ and $(T(\l+1) * I) \circ (I * T(\l-1)) \circ (T(\l+1) * I)$ must be non-zero multiples of each other or one of them is zero. The rest of the argument below follows formally from relations (iii). Note that we will not use relation (ii) in the proof of (iii).

\begin{Lemma} $(T(\l+1) * I) \circ (I * T(\l-1)) \ne 0$.
\end{Lemma}
\begin{proof}
We have (using shorthand notation)
\begin{align*}
X(\l-2) \circ T(\l+1) \circ T(\l-1) &= T(\l+1) \circ X(\l-2) \circ T(\l-1) \\
&= T(\l+1) \circ (T(\l-1) \circ X(\l)  - I) \\
&= -T(\l+1)  + T(\l+1) \circ T(\l-1) \circ X(\l)
\end{align*}
where we use that $X(\l-1) =  I * I * X(\l-1)$ and $T(\l+1) = T(\l+1) * I$ commute to get the first equality. So if $T(\l+1) \circ T(\l-1) = 0$ then $T(\l+1) = 0$ (contradiction).
\end{proof}

Similar to above we obtain
\begin{equation} \label{eq:XsTs1}
X(\l-2) \circ T(\l-1) \circ T(\l+1) \circ T(\l-1)
= T(\l-1) \circ T(\l+1) \circ  T(\l-1) \circ X(\l+2) - T(\l+1) \circ T(\l-1).
\end{equation}
Notice that this means $T(\l-1) \circ T(\l+1) \circ T(\l-1) \ne 0$ because $T(\l+1) \circ T(\l-1) \ne 0$.

Again by similar manipulations we obtain,
\begin{equation} \label{eq:XsTs2}
X(\l-2) \circ T(\l+1) \circ T(\l-1) \circ T(\l+1)
= T(\l+1) \circ T(\l-1) \circ T(\l+1) * X(\l+2) - T(\l+1) \circ T(\l-1).
\end{equation}
Thus $T(\l+1) \circ T(\l-1) \circ T(\l+1) \ne 0$. So $ T(\l-1) \circ T(\l+1) \circ T(\l-1) = \mu (T(\l+1) \circ T(\l-1) \circ T(\l+1))$ for some $\mu \in \k^\times$. Combining (\ref{eq:XsTs1}) and (\ref{eq:XsTs2}) we obtain that $\mu (T(\l+1) \circ T(\l-1)) = T(\l+1) \circ T(\l-1) \ne 0$ so $\mu = 1$ and we are done.

\subsection{Some final isomorphisms}

Having proved Theorem \ref{th:nilHeckerels} we need to finish the proof of the main Theorem \ref{th:main} by checking that certain maps induce isomorphisms. 

The first of these is that 
$$\oplus_{i=0}^r (X(\l + r)^i * I) \circ \iota \la-2i\ra : \sE^{(r+1)}(\l) \otimes_\k H^\star(\p^r) \rightarrow \sE(\l+r) \circ \sE^{(r)}(\l-1)$$
and
$$\oplus_{i=0}^r \pi \la2i\ra * (X(\l+r)^i \circ I) : \sE(\l+r) \circ \sE^{(r)}(\l-1) \rightarrow \sE^{(r+1)}(\l) \otimes_\k H^\star(\p^r)$$
induce isomorphisms. Fortunately, the first isomorphism follows immediately from Proposition \ref{prop:theta_k} because
$$X(\l+r) * I = (a(\l+r) \theta(\l+r+1), b(\l+r) \theta(\l+r-1), 0)$$
induces the same map (up to non-zero multiple) on cohomology as $\Theta(\l-1+r) = (0, \theta(\l-1+r), 0)$ (here we used that $b(\l+r) = -a(\l+r-2) \ne 0$). The proof of the second isomorphism is the same. 

The second thing we need to check is that for $\l \le 0$
$$\sigma + \sum_{j=0}^{-\l-1} (I * X(\l+1)^j \la-2j\ra) \circ \eta: \sE(\l-1) * \sF(\l-1) \oplus \O_\Delta \otimes_{\k} H^\star(\p^{-\l-1}) \xrightarrow{\sim} \sF(\l+1) * \sE(\l+1)$$
induces an isomorphism (and similarly if $\l \ge 0$). Now
\begin{equation}\label{eq:2}
\sum_{j=0}^{-\l-1} (I * X(\l+1)^j \la-2j\ra) \circ \eta : \O_\Delta \otimes_{\k} H^\star(\p^{-\l-1}) \rightarrow \sF(\l+1) * \sE(\l+1) 
\end{equation}
induces an isomorphism $\O_\Delta \otimes_{\k} H^\star(\p^{-\l-1}) \rightarrow \sP$. To see this note that by Proposition \ref{prop2:EF=FE}, $I * (\theta(\l+2),0)$ acting on $\sF(\l+1) * \sE(\l+1)$ induces a map 
\begin{equation}\label{eq:1}
\sE(\l-1) * \sF(\l-1) \oplus \sP \rightarrow \sE(\l-1) * \sF(\l-1) [2] \oplus \sP[2]
\end{equation}
which on $\sP$ is an isomorphism (at the level of cohomology) between all but two copies of $\O_\Delta$. Since $I * (0, \theta(\l))$ induces zero at the level of cohomology the map 
$$I * X(\l+1) = I * (a \theta(\l+2), b \theta(\l))$$
induces the same map as (\ref{eq:1}) at the level of cohomology. Thus, as $j$ varies, $(I * X(\l+1)^j) \circ \eta$ maps $\O_\Delta$ onto every copy of $\O_\Delta$ in $\H^*(\sP)$. Thus, summing over all $j$ we get a quasi-isomorphism, which is isomorphism (\ref{eq:2}) in the derived category. 

It remains to show that 
$$\sigma: \sE(\l-1) * \sF(\l-1) \rightarrow \sF(\l+1) * \sE(\l+1) \cong \sE(\l-1) * \sF(\l-1) \oplus \sP$$
induces the zero map $\sE(\l-1) * \sF(\l-1) \rightarrow \sP$ and an isomorphism on $\sE(\l-1) * \sF(\l-1)$. The first claim follows from Lemma \ref{lem:FEtoP}. 

To see the second claim it suffices to show that $\sigma \ne 0$ because of Lemma \ref{lem:EFmapEF}.  To see that $\sigma \ne 0$ we look at the composition 
\begin{eqnarray*}
(I * X(\l+1)) \circ \sigma 
&=& (I * X(\l+1)) \circ (I * I * \epsilon) \circ (I * T(\l) * I) \circ (\eta * I * I) \\
&=& (I * I * \epsilon) \circ (I * X(\l+1) * I * I) \circ (I * T(\l) * I) \circ (\eta * I * I) \\
&=& (I * I * \epsilon) \circ \left(I + (I * T(\l) * I) \circ (I * I * X(\l-1) * I) \right) \circ (\eta * I * I) \\
&=& (I * I * \epsilon) \circ (\eta * I * I) + \sigma \circ (X(\l-1) * I)
\end{eqnarray*}
where we used nil affine Hecke relation (iii) to get the second last line. Now if $\sigma$ induces zero then this means that 
$$(I * I * \epsilon) \circ (\eta * I * I) = \eta \circ \epsilon: \sE(\l-1) * \sF(\l-1) \rightarrow \sF(\l+1) \circ \sE(\l+1)\la2\ra$$
induces zero. But this map is the composition 
$$\sE(\l-1) * \sF(\l-1) \xrightarrow{\epsilon} \O_\Delta \la-\l+1\ra \xrightarrow{\eta} \sF(\l+1) * \sE(\l+1) \la2\ra$$
where the second map is an inclusion of $\O_\Delta \la-\l+1\ra$ into the bottom of $\O_\Delta \otimes_\k H^\star(\p^{-l-1})\la2\ra$. Thus $\eta \circ \epsilon \ne 0$ and thus $\sigma \ne 0$. 

\begin{Lemma}\label{lem:EFmapEF} If $\l \le 0$ then 
\begin{equation*}
\Hom(\sE(\l-1) * \sF(\l-1), \sE(\l-1) * \sF(\l-1)[i]\{j\}) =  
\begin{cases} 
0  \text{ if } i < 0 \text{ or } i = 0 \ne j \\
\k \cdot I \text{ if } i = 0 = j 
\end{cases}
\end{equation*}
while if $\l \ge 0$ then 
\begin{equation*}
\Hom(\sF(\l+1) * \sE(\l+1), \sF(\l+1) * \sE(\l+1)[i]\{j\}) = 
\begin{cases} 
0  \text{ if } i < 0  \text{ or } i = 0 \ne j \\
\k \cdot I \text{ if } i = 0 = j.
\end{cases}
\end{equation*}
\end{Lemma}
\begin{proof}
Suppose $\l \le 0$. Then 
\begin{align*}
&\Hom(\sE(\l-1) * \sF(\l-1), \sE(\l-1) * \sF(\l-1)[i]\{j\}) \\ 
&\cong \Hom(\sF(\l-1), \sF(\l-1) * \sE(\l-1) * \sF(\l-1) [\l-1+i]\{-\l+1+j\}) \\
&\cong \Hom(\sF(\l-1)[-\l+1-i]\{\l-1-j\}, \sE(\l-3) * \sF(\l-3) * \sF(\l-1) \oplus \sF(\l-1) \otimes_\k H^\star(\p^{-\l+1})). 
\end{align*}
Now $\Hom(\sF(\l-1)[-\l+1-i]\{\l-1-j\}, \sF(\l-1) \otimes_\k H^\star(\p^{-\l+1}))$ is zero if $i < 0$ or $i = 0 \ne j$ and $\k$ if $i=0=j$. Also,
\begin{align*}
& \Hom(\sF(\l-1)[-\l+1-i]\{\l-1-j\}, \sE(\l-3) * \sF(\l-3) * \sF(\l-1)) \\
&\cong \Hom(\sF(\l-3)[-\l+3]\{\l-3\} * \sF(\l-1)[-\l+1-i]\{\l-1-j\}, \sF(\l-3) * \sF(\l-1)) \\
&\cong \Hom(\sF^{(2)}(\l-2) \otimes_\k H^\star(\p^1), \sF^{(2)}(\l-2) \otimes_\k H^\star(\p^1)[2\l-4+i]\{-2\l+4+j\})
\end{align*}
is zero since $2\l-4+i < -2$. The result follows. 

The case $\l \ge 0$ is proved similarly. 
\end{proof}

\end{document}